\documentclass[11pt, reqno]{amsart}

\usepackage{amssymb, amscd, latexsym}
\usepackage[utf8]{inputenc}
\usepackage{enumitem}
\usepackage{amsmath}
\usepackage{amsthm}
\usepackage{graphicx}
\usepackage[all]{xy}
\usepackage{float}
\usepackage{mathrsfs}
\usepackage{pgf,tikz}
\usetikzlibrary{arrows}
\usepackage{array}
\usepackage{arydshln}
\usepackage{hyperref}
\hypersetup{
pdftitle={Decomposable Clutters},
pdfauthor={A.Yazdan Pour, M. Bigdeli, R. Zaare-Nahandi},
pdfsubject={Decomposable Clutters},
pdfcreator={A.Yazdan Pour, M. Bigdeli, R. Zaare-Nahandi},
colorlinks,
linkcolor=blue,
citecolor=magenta,
anchorcolor=red,
bookmarksopen,
urlcolor=red,
filecolor=red,
pdfpagetransition={Wipe}
}

\theoremstyle{plain}
\newtheorem{thm}{Theorem}[section]
\newtheorem{cor}[thm]{Corollary}
\newtheorem{lem}[thm]{Lemma}
\newtheorem{prop}[thm]{Proposition}
\newtheorem{conj}[thm]{Conjecture}
\newtheorem*{conj*}{Conjecture}

\theoremstyle{definition}
\newtheorem{defn}[thm]{Definition}
\newtheorem{ex}[thm]{Example}
\newtheorem{rem}[thm]{Remark}

\theoremstyle{remark}

\def\cocoa{{\hbox{\rm C\kern-.13em o\kern-.07em C\kern-.13em o\kern-.15em A}}}

\def\implies{\ifmmode\Rightarrow \else
        \unskip${}\Rightarrow{}$\ignorespaces\fi}

\def\supp{\mathrm{supp}}
\def\gcd{{\rm gcd}}
\let\sect=\cap
\def\xb{{\mathbf{x}}}
\def\eb{{\mathbf{e}}}
\def\C{\mathcal{C}}

\begin{document}

\title[Decomposable clutters \& Simon's conjecture]{Decomposable clutters and a generalization of Simon's conjecture}
\author[M. Bigdeli, A. A. Yazdan Pour, R. Zaare-Nahandi]{mina bigdeli, {ali akbar} {yazdan pour}, rashid zaare-nahandi}
\address{Mina Bigdeli, School of Mathematics\\ Institute for Research in Fundamental Sciences (IPM)\\  P.O.Box: 19395-5746\\ Tehran, Iran}
		\email{mina.bigdeli98@gmail.com, mina.bigdeli@ipm.ir}
\address{Ali Akbar Yazdan Pour,  Department of mathematics\\ institute for advanced studies in basic sciences (IASBS)\\ P.O.Box 45195-1159 \\ Zanjan, Iran}
\email{yazdan@iasbs.ac.ir}
\address{Rashid Zaare-Nahandi,  Department of mathematics\\ institute for advanced studies in basic sciences (IASBS)\\ P.O.Box 45195-1159 \\ Zanjan, Iran}
\email{rashidzn@iasbs.ac.ir}
\thanks{Bigdeli's research was supported by a grant from IPM}

\begin{abstract}
Each (equigenerated) squarefree monomial ideal in the polynomial ring $S=\mathbb{K}[x_1, \ldots, x_n]$ represents a family of subsets of $[n]$, called a (uniform) clutter. In this paper, we introduce a class of uniform clutters, called decomposable clutters, whose associated ideal has linear quotients and hence linear resolution over all fields. We  show that chordality of these clutters guarantees the correctness of a conjecture raised by R.~S.~Simon \cite{Simon} on extendable shellability of $d$-skeletons of a simplex $\langle [n] \rangle$, for all $d$. We then prove this conjecture for $d \geq n-3$.  
\end{abstract}
\maketitle


\section*{introduction}
In the study of simplicial complexes, shellability is one of the interesting and widely considered topics. However, it is not easy to determine whether a simplicial complex is shellable. There are some known classes of such complexes arising from  different structures. In this paper, we introduce a new class of shellable complexes which arise from decomposable clutters. 

Shellability is a simple and powerful combinatorial tool for obtaining sequentially Cohen-Macaulay property. Moreover, shellability is one of the most important tools for polytopes to satisfy  Euler-Poincar\'e formula. Recall that the Euler-Poincar\'e  formula states that for a $d$-dimensional polytope $P$, one has
$$\sum\limits_{i=-1}^d (-1)^i f_i=0,$$
where $f_i$ denotes the number of $i$-faces of $P$ (with $f_{-1}=f_d=1$).
For a historical review of the importance and motivation behind the notion of shellability, the reader may refer to state of the art paper by J. Gallier \cite[pp. 111-112]{Gallier}.

In \cite{Xavier} it is proved that for every $d \geq 2$, deciding if a pure $d$-dimensional simplicial complex is shellable is NP-hard, hence NP-complete. So it is of great interest to find some classes of simplicial complexes which are shellable. Some known results in this area are as follows:
\begin{itemize}
\item Every skeleton of a shellable simplicial complex is shellable \cite[Theorem 2.9]{Bjorner-Wachs}. In particular, every skeleton of a simplex is shellable. 
\item Vertex decomposable simplicial complexes are shellable (essentially \cite{Wachs2}). In particular, matroid complexes are shellable \cite{Provan}.
\item If $P$ is a polytope, then the boundary complex of $P$ is shellable \cite{Brugesser}.
\item If $G$ is a chordal graph, then the independence complex of $G$ is shellable \cite[Theorem 2.13]{Vantuyl}.
\item If $G$ is a chordal graph, and $\Delta=\Delta(G)$ is the clique complex of $G$, then the Stanley--Reisner ideal $I_\Delta$ has a linear resolution \cite{Fr}. So by \cite[Theorem 3.2]{HHZ} the ideal $I_\Delta$ has linear quotients. Hence the Alexander dual of $\Delta$ is shellable \cite[Theorem 1.4(c)]{HHZ}.
\end{itemize}
One of the main results of this paper concerns a generalization of chordal graphs to hypergraphs, called decomposable clutters, with the property that the Alexander dual of the clique complex is shellable (Corollary~\ref{shellability of simp. of decomposables}).

Another important class of shellable simplicial complexes is  the class of extendably shellable simplicial complexes. A simplicial complex $\Delta$ is called \textit{extendably shellable}, if any shelling of a subcomplex of $\Delta$ can be continued to be a shelling of $\Delta$. By a subcomplex of $\Delta$, here we mean a simplicial complex $\Gamma$ whose facets are facets of $\Delta$.
As in the case of shellable simplicial complexes, it seems to be quite difficult to show whether a special class of complexes is extendably shellable. It is known that any 2-sphere is extendably shellable \cite[p. 37]{Klee}. H. Tverberg has asked whether, for $d \geq 3$, each convex $d$-sphere is extendably shellable (see \cite{Klee2}).  Later, Ziegler \cite{Ziegler} showed that there are simple and simplicial polytopes whose boundary complex is not extendably shellable.  
This fact  gave a negative answer to the question of Tverberg.  However, in \cite{Kleinschmidt}, it is shown that each $d$-sphere with $d + 3$ vertices is extendably shellable (see \cite[p. 49]{Klee}). An intriguing conjecture due to R.~S.~Simon~\cite[Conjecture 4.2.1]{Simon} is the following:

\begin{conj*}[{Simon's Conjecture}] 
Every $d$-skeleton of a simplex is extendably shellable.
\end{conj*}

Bj\"orner and Eriksson in \cite{Bjorner} have proved that any matroid of rank $3$ is extendably shellable. Since $2$-skeleton of a simplex is a matroid of rank $3$,  Simon's conjecture holds for $d=2$. In \cite[Remark 1]{Bjorner}, as a natural strenthening of Simon's conjecture, the authors asked if all matroidal simplicial complexes are extendably shellable. Hall in \cite{Hall} presented a matroid of rank $12$ which is not extendably shellable. So this is a counterexample to the extended conjecture of Bj\"orner and Eriksson.

As one of the main results of this paper, in Corollary~\ref{partial answer to Simon's conjecture}, we will show that the $d$-th skeleton   of the simplex $\langle [n] \rangle$  is extendably shellable for $d \geq n-3$. Moreover, we make a stronger conjecture which is a generalization of Simon's conjecture, (see Conjecture~\ref{decomposables are chordal-conjecture}). Our approach to get a partial answer to Simon's conjecture is as follows:

In Section~\ref{prelim} we introduce algebraic and combinatorial backgrounds which will be used in this paper. In Section~\ref{decomposable}, we introduce a generalization of chordality from graphs to hypergraphs, which is called decomposability of clutters. In this section it is proved that the Stanley-Reisner ideal of the clique complex of decomposable clutters have linear quotients. Hence the Alexander dual of the clique complex of a decomposable clutter is shellable. Then in Section~\ref{Simon}, we study the relation between the concept of decomposable clutter and the Simon's conjecture. To be more precise, we consider the class of chordal clutters as introduced in \cite{BYZ}. The ideal associated to chordal clutters have a linear resolution over all fields, while there are examples of chordal clutters whose  associated ideal does not have linear quotients. Yet, the ideal attached to the class of decomposable clutters has linear quotients. It follows that the class of chordal clutters is different from the class of decomposable clutters. However, since the ideals associated to the class of decomposable clutters have a linear resolution over all fields, it is reasonable to ask whether this class is contained in the class of chordal clutters. We will see that this statement is a generalization of Simon's conjecture  (Corollary~\ref{Inverse of Simon's conjecture}).  We close the paper by giving some examples of classes of decomposable clutters in the last section.


\section{Preliminaries} \label{prelim}
Throughout this paper, $S=\mathbb{K}[x_1, \ldots, x_n]$ denotes the polynomial ring over a field $\mathbb{K}$ with $n$ variables, endowed with standard grading (i.e. $\deg(x_i) =1$). Let $I \neq 0$ be a  graded ideal of $S$ and
$$
 \cdots \to F_2 \to F_1 \to F_0 \to I \to 0,
$$
be a graded minimal free resolution of $I$ with $F_i = \oplus_j S(-j)^{\beta_{i,j}(I)}$, for all $i \geq 0$.

The numbers $\beta_{i,j}(I) = \dim_{\mathbb{K}} \mbox{Tor}^S_i(I, \mathbb{K})_j$ are called the \textit{graded Betti numbers} of $I$. The \textit{Castelnuovo-Mumford regularity} of $I$, $\mathrm{reg}(I)$, is given by
$$
\mbox{reg}(I) = \sup\{j - i \colon \quad \beta_{i,j}(I) \neq 0\}.
$$

We say that $I$ has a $d$-\textit{linear resolution} if $\beta_{i,j}(I) =0$ for all $i, j$ with $j-i>d$. If this is the case, then  $I$ is generated by homogeneous elements of degree $d$. In this paper we focus on non-zero homogeneous ideals that can be generated by squarefree monomials. Such ideals are called \textit{squarefree monomial ideals}.

We denote by $\mathcal{G}\left(I\right)$, the set of minimal generating set of a monomial ideal $I\subset S$. For two ideals $I, J \subset S$ the set $I \colon J=\{ f\in S \colon \; fg\in I \text{\ for all } g\in J\}$ is an ideal in $S$, called the {\em colon ideal} of $I$ with respect to $J$. The following is an easy consequence of the properties of monomial ideals in $S$.

\begin{prop}[{\cite[Proposition~1.2.2]{HHBook}}]\label{colon ideal}
Let $I$ and $J$ be monomial ideals. Then $I \colon J$ is a monomial ideal, and	
	$$I \colon J=\bigcap_{v\in \mathcal{G}\left(J\right)}I \colon \left( v \right).$$
Moreover, $\{u/ \gcd\left(u,v\right) \colon \; u \in \mathcal{G}\left(I\right)\}$ is a set of generators of $I \colon \left( v \right)$.
\end{prop}

A homogeneous ideal $I$ is said to have \textit{linear quotients}, if $I$ has an ordered set of minimal generators $\left\{ u_1, \ldots, u_r \right\}$ such that the colon ideal $\left( u_1, \ldots, u_{i-1} \right) \colon u_i$ is generated by linear forms, for $i= 2, \ldots, r$. If $I$ is an equigenerated ideal with linear quotients, then $I$ has a linear resolution \cite[Proposition~8.2.1]{HHBook}.


\subsection{Simplicial complexes}
A \textit{simplicial complex} $\Delta$ on the vertex set $V=\{v_1, \ldots, v_n\}$ is a collection of subsets of $V$ such that $\{ v_i \} \in \Delta$  for all $i$ and, $F \in \Delta$ implies that all subsets of $F$ are also in $\Delta$. The elements of $\Delta$ are called
\textit{faces} and the maximal faces under inclusion are called \textit{facets} of $\Delta$. We denote by $\mathcal{F}(\Delta)$ the set of facets of $\Delta$. By $\langle F_1,\ldots, F_t\rangle$ we mean the simplicial complex whose facets are $F_1,\ldots, F_t$. A simplicial complex which has only one facet is called a \textit{simplex}.  
A subset $F\subseteq [n]$ is called a \textit{non-face} of $\Delta$ if $F\notin \Delta$.

 The {\em dimension} of a face $F$ is $\dim F = |F|-1$, where $|F|$ denotes the cardinality of $F$. A simplicial complex is called {\em pure} if all its facets have the same dimension. 
The \textit{dimension} of $\Delta$, $\dim(\Delta)$, is defined as:
$$\dim (\Delta) = \max\{\dim F \colon F \in \Delta \}.$$

For a simplicial complex $\Delta$ of dimension $d$ and for $0\leq i\leq d$, the {\em $i$-th  skeleton} of $\Delta$, denoted by $\Delta^{(i)}$, is a simplicial complex  whose faces are all faces of $\Delta$ with dimension$\leq i$. By {\em pure $i$-th skeleton} of $\Delta$ we mean a simplicial complex $\Delta^{[i]}$ whose facets are all $i$-faces of $\Delta$. For a simplex these two concepts coincide.

Simplicial complexes are in one-to-one correspondence with squarefree monomial ideals.  To each simplicial complex $\Delta$ on the vertex set $[n]$ we associate a squarefree monomial ideal $I_\Delta\subset S$, which is called the Stanley-Reisner ideal of $\Delta$, defined as follows:
$$I_\Delta=(\mathbf{x}_F \colon \ F\notin \Delta ),$$
where $\mathbf{x}_F=\prod_{i\in F}x_i$, for $F \subset [n]$.

\begin{defn}[Shellable simplicial complexes]
A simplicial complex $\Delta$ is called \textit{shellable} if there is a total order of the facets of  $\Delta$, say $F_1, \ldots, F_t$, such that $\langle F_1, \ldots, F_{i-1} \rangle \cap \langle F_i \rangle$ is generated by a non-empty set of maximal proper faces of $F_i$ for $2 \leq i \leq t$. Any such order is called a \textit{shelling order} of $\Delta$.
\end{defn}

For a simplicial complex $\Delta$ on the vertex set $[n]$ and for a facet $F \in \Delta$, let $\bar{F} = [n] \setminus F$ be the complement of $F$. 
The \textit{Alexander dual} of $\Delta$, denoted by $\Delta^\vee$, is  the simplicial complex
$$\Delta^\vee = \{ \bar{F} \colon \quad F \notin \Delta \}.$$
The facets of $\Delta^\vee$ are the complements of minimal non-faces of $\Delta$. Moreover, $(\Delta^\vee)^\vee=\Delta$. Hence, if $\mathcal{F}(\Delta)=\{F_1,\ldots,F_t\}$, then $I_{\Delta^{\vee}}=\left( \xb_{\bar{F}_1}, \ldots, \xb_{\bar{F}_r} \right)$.

The following result shows that the facets of a shellable simplicial complex induces an order of linear quotients for an appropriate ideal.
\begin{prop}[{\cite[Proposition~ 8.2.5]{HHBook}}]\label{linear quotient and shellability}
Let $\Delta$ be a shellable simplicial complex on the vertex set $[n]$. The followings are equivalent:
\begin{itemize}
\item[\rm (i)] $F_1, \ldots, F_r$ is a shelling order of $\Delta$;
\item[\rm (ii)] The ideal $I_{\Delta^{\vee}}=\left( \xb_{\bar{F}_1}, \ldots, \xb_{\bar{F}_r} \right)$ has linear quotients with respect to  the given order.
\end{itemize}
\end{prop}


\subsection{Clutters}

In this part we recall some definitions about clutters and their associated ideals. 
\begin{defn}[Clutter] \label{SC} 
A \textit{clutter} $\mathcal{C}$ on the vertex set $[n]$ is a collection of subsets of $[n]$, called \textit{circuits} of $\mathcal{C}$, such that if $F_1$ and $F_2$ are distinct circuits, then $F_1 \nsubseteq F_2$.
A \textit{$d$-circuit} is a circuit consisting of exactly $d$ vertices, and a clutter is called \textit{$d$-uniform} if every circuit has $d$ vertices.
\end{defn}

A subclutter of a clutter $\mathcal{C}$ is a subset of $\mathcal{C}$. If $\mathcal{C}$ is a clutter on the vertex set $[n]$ and $W \subseteq [n]$, then the \textit{induced subclutter} of $\mathcal{C}$ on $W$, $\mathcal{C}\lceil_W$, is defined as:
$$\mathcal{C}\lceil_W= \{F \in \mathcal{C} \colon \quad F \subseteq W \}.$$

For a non-empty clutter $\mathcal{C}$ on the vertex set $[n]$, we define the ideal $I \left( \mathcal{C} \right)$, as follows:
$$I(\mathcal{C}) = \left(  \textbf{x}_T \colon \quad T \in \mathcal{C} \right),$$
and we define $I(\varnothing) = 0$. The ideal $I \left( {\mathcal{C}} \right)$ is called the \textit{circuit ideal} of $\mathcal{C}$. 

Let $n$, $d$ be positive integers. For $n \geq d$, by we mean the complete clutter on the vertex set $V$ with $|V|=n$, that is
$$\mathcal{C}_{n,d} = \left\{ F \subseteq V \colon \quad |F|=d \right\}.$$
This clutter is called  the \textit{complete $d$-uniform clutter} on $V$ with $n$ vertices.
In the case that $n<d$, we let $\mathcal{C}_{n,d}$ be some isolated points. It is well-known that for $n \geq d$ the ideal $I \left( \mathcal{C}_{n,d} \right)$ has a $d$-linear resolution (see e.g. \cite[Example 2.12]{MNYZ}). 

If $\mathcal{C}$ is a $d$-uniform clutter on $[n]$, we define $\bar{\mathcal{C}}$, the \textit{complement} of $\mathcal{C}$, to be
\begin{equation*}
\bar{\mathcal{C}} = \mathcal{C}_{n,d} \setminus \mathcal{C} = \{F \subseteq [n] \colon \quad |F|=d, \,F \notin \mathcal{C}\}.
\end{equation*}

Frequently in this paper, we take a $d$-uniform clutter $\mathcal{C} \neq \mathcal{C}_{n,d}$ on the vertex set $[n]$ and consider the squarefree monomial ideal $I=I(\bar{\mathcal{C}})$ in the polynomial ring $S=\mathbb{K}[x_1, \ldots, x_n]$. 


\subsection{Chordal clutters}
In the following we recall some definitions and concepts from \cite{BYZ}.

Let $\mathcal{C}$ be a $d$-uniform clutter on the vertex set $[n]$ and let $\Delta (\C)$ be the simplicial complex on the vertex set $[n]$ with $I_{\Delta (\C)} = I \left( \bar{\C} \right)$. The simplicial complex $\Delta (\C )$ is called the \textit{clique complex} of $\C$ and a face $F \in \Delta (\C)$ is called a \textit{clique} in $\C$. It is easily seen that $F \subseteq [n]$ is a clique in $\C$ if and only if either $|F|<d$ or else all $d$-subsets of $F$ belongs to $\C$.

For any $(d-1)$-subset $e$ of $[n]$ and a $d$-uniform clutter $\mathcal{C}$, let
\begin{equation*}
{N}_{\mathcal{C}} \left[ e \right] = e \cup \lbrace c \in \left[ n \right] \colon \quad e \cup \lbrace c \rbrace \in \mathcal{C} \rbrace.
\end{equation*}
We call ${N}_{\mathcal{C}} \left[ e \right]$ the \textit{closed neighborhood} of $e$ in $\mathcal{C}$.  In the case that $e \neq N_{\mathcal{C}} \left[ e \right]$ (i.e. $e \subset F$, for some $F \in \mathcal{C}$), $e$ is called a \textit{maximal subcircuit} of $\mathcal{C}$. The set of all maximal subcircuits of $\mathcal{C}$ is denoted by ${\rm SC}\left(\mathcal{C}\right)$. We say that $e$ is \textit{simplicial} over $\mathcal{C}$, if ${N}_{\mathcal{C}} \left[ e \right] \in \Delta(\mathcal{C})$. One may note that a $(d-1)$-subset of $[n]$ which is not a maximal subcircuit is simplicial over $\mathcal{C}$. If $e \in \mathrm{SC} \left(\mathcal{C}\right)$ and  $e$ is simplicial over $\mathcal{C}$, then $e$ is called a \textit{simplicial maximal subcircuit} of $\mathcal{C}$. Let us denote by $\mathrm{Simp} \left( \mathcal{C} \right)$, the set of all $(d-1)$-subsets of $[n]$ which are simplicial over $\mathcal{C}$.

Let $\mathcal{C}$ be a clutter and let $e$ be a subset of $[n]$. By $\mathcal{C} \setminus e$ we mean the clutter 
$$ \left\{F \in \mathcal{C} \colon \ e \nsubseteq F \right\}.$$
This clutter is called the \textit{deletion of $e$ from} $\mathcal{C}$.

\begin{defn}[{\cite[Definition 3.1]{BYZ}}] \label{def of chordal}
Let $\mathcal{C}$ be a $d$-uniform clutter. We call $\mathcal{C}$ a \textit{chordal clutter}, if either $\mathcal{C}=\varnothing$, or $\mathcal{C}$ admits a simplicial maximal subcircuit $e$ such that $\mathcal{C}\setminus e$ is chordal.
\end{defn}
Following the notation in \cite{BYZ}, we use $\mathfrak{C}_d$ to denote the class of all $d$-uniform chordal clutters. 

\begin{defn}
Let $\mathcal{C}$ be a $d$-uniform clutter. A sequence of  $(d-1)$-subsets of $[n]$, say $\eb= e_1, \ldots, e_r$, is called a \textit{simplicial sequence} in $\mathcal{C}$, if $e_1$ is  simplicial over $\mathcal{C}$ and $e_i$ is simplicial over $\left( \left( \left( \mathcal{C} \setminus e_{1} \right) \setminus e_2 \right)  \setminus \cdots \right) \setminus e_{i-1}$ for all $1 <i \leq r$.
\end{defn}

As seen from the definition, the $d$-uniform clutter $\mathcal{C}$ is chordal if either $\mathcal{C} = \varnothing$, or else there exists a sequence of maximal subcircuits of $\mathcal{C}$, say $\eb=e_1, \ldots, e_t$, such that $e_1$ is simplicial maximal subcircuit over $\mathcal{C}$, $e_i$ is simplicial maximal subcircuit over $\left( \left( \left( \mathcal{C} \setminus e_{1} \right) \setminus e_2 \right)  \setminus \cdots \right) \setminus e_{i-1}$ for all $i>1$, and $\left( \left( \left(  \mathcal{C} \setminus e_{1} \right) \setminus e_2 \right) \setminus \cdots \right) \setminus e_{t} = \varnothing$.  The sequence $\eb$ is called a {\em simplicial order} of $\mathcal{C}$.

To simplify the notation, given a $d$-uniform clutter $\mathcal{C}$ and a simplicial sequence $\textbf{e} = e_1, \ldots, e_r$ in $\mathcal{C}$, we use $\mathcal{C}_\textbf{e}^0$ for  $\mathcal{C}$ and $\mathcal{C}_\textbf{e}^i $ for $\left( \mathcal{C} \setminus e_1 \right) \setminus \cdots \setminus e_{i}$ for $i\geq 1$. 

In \cite[Remark~2]{BYZ} it is mentioned that Definition~\ref{def of chordal} coincides with the graph theoretical definition of chordal graphs in the case $d=2$.

\begin{ex}
In Figure~\ref{F2}, the clutter $\mathcal{C}$ is chordal, while the clutter $\mathcal{D}$ is not.
	\begin{align*}
	&\mathcal{C}=\{ \{ 1,2,3 \}, \{ 1,2,4 \}, \{ 1,3,4 \}, \{ 2,3,4 \}, \{ 1,2,5 \},
	\{ 1,2,6 \}, \{ 1,5,6 \}, \{ 2,5,6 \} \}. \\
	&\mathcal{D}=\{\{1,2,3\}, \{1,2,4\}, \{1,3,4\},\{2,3,5\}, \{2,4,5\}, \{3,4,5\}\}.
	\end{align*}
	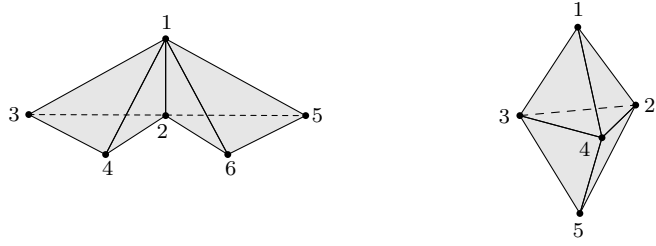
\begin{figure}[H]
		\begin{tikzpicture}[line cap=round,line join=round,>=triangle 45,x=0.7cm,y=0.7cm]
		\clip(5.26,2.0) rectangle (12.44,6.8);
		\fill[fill=black,fill opacity=0.1] (8.8,5.94) -- (6.2,4.5) -- (7.66,3.74) -- cycle;
		\fill[fill=black,fill opacity=0.1] (8.8,5.94) -- (8.8,4.48) -- (7.66,3.74) -- cycle;
		\fill[fill=black,fill opacity=0.1] (8.8,5.94) -- (11.46,4.48) -- (9.98,3.74) -- cycle;
		\fill[fill=black,fill opacity=0.1] (8.8,5.94) -- (8.8,4.48) -- (9.98,3.74) -- cycle;
		\draw (8.8,5.94)-- (6.2,4.5);
		\draw (6.2,4.5)-- (7.66,3.74);
		\draw (7.66,3.74)-- (8.8,5.94);
		\draw (8.8,5.94)-- (8.8,4.48);
		\draw (8.8,4.48)-- (7.66,3.74);
		\draw (7.66,3.74)-- (8.8,5.94);
		\draw (8.8,5.94)-- (11.46,4.48);
		\draw (11.46,4.48)-- (9.98,3.74);
		\draw (9.98,3.74)-- (8.8,5.94);
		\draw (8.8,5.94)-- (8.8,4.48);
		\draw (8.8,4.48)-- (9.98,3.74);
		\draw (9.98,3.74)-- (8.8,5.94);
		\draw [dash pattern=on 2pt off 2pt] (6.2,4.5)-- (11.46,4.48);
		\draw (8.52,6.58) node[anchor=north west] {\begin{scriptsize}1\end{scriptsize}};
		\draw (8.44,4.50) node[anchor=north west] {\begin{scriptsize}2\end{scriptsize}};
		\draw (5.62,4.83) node[anchor=north west] {\begin{scriptsize}3\end{scriptsize}};
		\draw (7.4,3.8) node[anchor=north west] {\begin{scriptsize}4\end{scriptsize}};
		\draw (11.38,4.82) node[anchor=north west] {\begin{scriptsize}5\end{scriptsize}};
		\draw (9.74,3.8) node[anchor=north west] {\begin{scriptsize}6\end{scriptsize}};
		\begin{scriptsize}
		\fill [color=black] (8.8,5.94) circle (1.3pt);
		\fill [color=black] (6.2,4.5) circle (1.3pt);
		\fill [color=black] (7.66,3.74) circle (1.3pt);
		\fill [color=black] (8.8,4.48) circle (1.3pt);
		\fill [color=black] (11.46,4.48) circle (1.3pt);
		\fill [color=black] (9.98,3.74) circle (1.3pt);
		\end{scriptsize}
		\end{tikzpicture}
		\quad
		\quad
		\quad
		\begin{tikzpicture}[line cap=round,line join=round,>=triangle 45,x=.7cm,y=.7cm]
		\clip(13.89,2.00) rectangle (18.76,6.8);
		\fill[fill=black,fill opacity=0.1] (16.24,6.16) -- (15.14,4.48) -- (16.7,4.06) -- cycle;
		\fill[fill=black,fill opacity=0.1] (16.24,6.16) -- (17.34,4.68) -- (16.7,4.06) -- cycle;
		\fill[fill=black,fill opacity=0.1] (15.14,4.48) -- (16.28,2.62) -- (16.7,4.06) -- cycle;
		\fill[fill=black,fill opacity=0.1] (17.34,4.68) -- (16.28,2.62) -- (16.7,4.06) -- cycle;
		\draw (16.24,6.16)-- (15.14,4.48);
		\draw (15.14,4.48)-- (16.7,4.06);
		\draw (16.7,4.06)-- (16.24,6.16);
		\draw (16.24,6.16)-- (17.34,4.68);
		\draw (17.34,4.68)-- (16.7,4.06);
		\draw (16.7,4.06)-- (16.24,6.16);
		\draw [dash pattern=on 3pt off 3pt] (15.14,4.48)-- (17.34,4.68);
		\draw (15.14,4.48)-- (16.28,2.62);
		\draw (16.28,2.62)-- (16.7,4.06);
		\draw (16.7,4.06)-- (15.14,4.48);
		\draw (17.34,4.68)-- (16.28,2.62);
		\draw (16.28,2.62)-- (16.7,4.06);
		\draw (16.7,4.06)-- (17.34,4.68);
		\draw (15.95,6.83) node[anchor=north west] {\begin{scriptsize}1\end{scriptsize}};
		\draw (17.31,5) node[anchor=north west] {\begin{scriptsize}2\end{scriptsize}};
		\draw (14.55,4.8) node[anchor=north west] {\begin{scriptsize}3\end{scriptsize}};
		\draw (16.07,4.17) node[anchor=north west] {\begin{scriptsize}4\end{scriptsize}};
		\draw (15.95,2.62) node[anchor=north west] {\begin{scriptsize}5\end{scriptsize}};
		\begin{scriptsize}
		\fill [color=black] (16.24,6.16) circle (1.3pt);
		\fill [color=black] (15.14,4.48) circle (1.3pt);
		\fill [color=black] (16.7,4.06) circle (1.3pt);
		\fill [color=black] (17.34,4.68) circle (1.3pt);
		\fill [color=black] (16.28,2.62) circle (1.3pt);
		\end{scriptsize}
		\end{tikzpicture}
		\caption{The Clutter $\mathcal{C}$ on the left and $\mathcal{D}$ on the right}
		\label{F2}
	\end{figure}

Indeed, while the clutter $\mathcal{D}$ does not have any simplicial maximal subcircuit, one of the possible simplicial orders for $\C$ is the following:
\begin{center}
\begin{tabular}{lll}
$e_1 = \{ 1, 3 \}$ & $e_2 = \{ 1, 4 \}$ & $e_3 = \{ 2, 4 \}$\\
$e_4 = \{ 1, 2 \}$ & $e_5 = \{ 2, 6 \}$ & $e_6 = \{ 1, 5 \}$\\
\end{tabular}
\end{center}
\end{ex}

A celebrated theorem of Fr\"oberg \cite{Fr} gives a complete characterization of squarefree monomial ideals generated in degree $2$ with linear resolution. 
\begin{thm}[Fr\"oberg's theorem]\label{Froberg}
Let $\mathcal{C}$ be a $2$-uniform clutter (i.e. a graph) and $I=I\left( \bar{\mathcal{C}} \right)$. Then $I$ has a linear resolution if and only if $\mathcal{C}$ is chordal.
\end{thm}


\subsection{Simplicial subclutters}\label{simp sub}
\cite[Definition~2.1]{MA}: Let $\mathcal{C}$ be a $d$-uniform clutter on the vertex set $[n]$ and $\mathcal{D} \subsetneq \mathcal{C}$ be a subclutter of $\mathcal{C}$. We say that $\mathcal{D}$ is a \textit{simplicial subclutter} of $\mathcal{C}$, if there exists a sequence of $(d-1)$-subsets of $[n]$, say $\textbf{e}= e_1, \ldots, e_t$, and $A_i \subseteq \left\{ F \in \mathcal{C}^{i-1}_\textbf{e} \colon \; e_i \subset F \right\}$, $i=1, \ldots, t$, such that 
\begin{itemize}
\item[(i)] $e_1$ is simplicial over $\mathcal{C}$;
\item[(ii)] $e_i$ is simplicial over $\mathcal{C} \setminus A_1 \setminus \cdots \setminus A_{i-1}$, for $i>1$;
\item[(iii)] $\mathcal{D} = \mathcal{C} \setminus A_1 \setminus \cdots \setminus A_t$.
\end{itemize}

\begin{thm}[{\cite[Corollary 2.5(c) and Corollary 2.9]{MA}}] \label{simplicial subclutter has linear quotients-cor}
Let $\mathcal{C}$ be a $d$-uniform clutter and let $\mathcal{D}$ be a simplicial subclutter of $\mathcal{C}$.
\begin{itemize}
\item[\rm (i)] $I\left(\bar{\mathcal{C}}\right)$ has a linear resolution if and only if $I\left(\bar{\mathcal{D}}\right)$ has a linear resolution.
\item[\rm (ii)] If $I\left(\bar{\mathcal{C}}\right)$ has linear quotients then  $I\left(\bar{\mathcal{D}}\right)$ has linear quotients.
\end{itemize}
\end{thm}

\begin{rem}
One may easily check that a simplicial subgraph of a chordal graph is again a chordal graph. However, it is not known that whether any simplicial subclutter of a chordal clutter is again chordal.
\end{rem}

\begin{rem}
If $\mathcal{C}$ is a $d$-uniform chordal clutter, then $\varnothing$ is a simplicial subclutter of $\mathcal{C}$. Hence by Theorem~\ref{simplicial subclutter has linear quotients-cor}(i) the ideal $I(\bar{\mathcal{C}})$ has a linear resolution (over all fields). This gives a generalization of Fr\"oberg's theorem in one direction.
\end{rem}


\section{Decomposable clutters} \label{decomposable}
In this section, we introduce a class of  uniform clutters, called \textit{decomposable} clutters, whose  associated ideals have linear quotients. We show by an example that this class is not equivalent to the class of equigenerated ideals with linear quotients. However, in the next section, we show that chordality of such clutters is indeed a generalization of Simon's conjecture.

\begin{defn}[Decomposable clutter]
A  \textit{decomposable clutter} is a  $d$-uniform clutter obtained  recursively as follows:
\begin{itemize}
\item[\rm(i)] $\mathcal{C}_{n,d}$  is a  decomposable clutter;
\item[\rm(ii)] If $\mathcal{D}=\mathcal{C}_1 \cup \mathcal{C}_2$, where $\mathcal{C}_1, \mathcal{C}_2$ are decomposable clutters on the vertex sets  $V\left(\mathcal{C}_1\right)$, $V\left(\mathcal{C}_{2}\right)$, respectively,  with the property that $V(\mathcal{C}_1) \nsubseteq V(\mathcal{C}_2)$, $V(\mathcal{C}_2) \nsubseteq V(\mathcal{C}_1)$ and  $V\left(\mathcal{C}_1\right)\cap V\left(\mathcal{C}_{2}\right)$ is a clique in both $\mathcal{C}_1, \mathcal{C}_2$, then $\mathcal{D}$ is  decomposable;
\item[\rm(iii)] If $\mathcal{C}$ is a  decomposable clutter, then every simplicial subclutter of $\mathcal{C}$ is decomposable.
\end{itemize}
We denote by $\mathfrak{C'}_d$, the class of all $d$-uniform decomposable clutters.
\end{defn}

\begin{ex}
Let $\mathcal{C}$ be the $3$-uniform clutter shown in Figure~\ref{example of SC}.
$$\mathcal{C}= \big\{ \{1,2,3\},\{1,2,4\},\{1,3,4\}, \{4,5,6\},\{4,5,7\},\{4,6,7\},\{5,6,7\} \big\}.$$
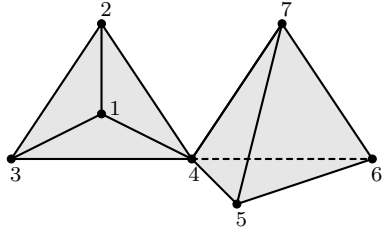
\begin{figure}[H]
\begin{center}
\begin{tikzpicture}[line cap=round,line join=round,>=triangle 45,x=0.3cm,y=0.3cm]
\clip(3.3,3) rectangle (21.5,14.5);
\fill[line width=1.2pt,fill=black,fill opacity=0.10000000149011612] (8.,12.) -- (4.,6.) -- (12.,6.) -- cycle;
\fill[line width=1.2pt,fill=black,fill opacity=0.10000000149011612] (16.,12.) -- (14.,4.) -- (12.,6.) -- cycle;
\fill[line width=1.2pt,fill=black,fill opacity=0.10000000149011612] (16.,12.) -- (14.,4.) -- (20.,6.) -- cycle;
\draw [line width=0.8pt] (8.,12.)-- (4.,6.);
\draw [line width=0.8pt] (4.,6.)-- (12.,6.);
\draw [line width=0.8pt] (12.,6.)-- (8.,12.);
\draw [line width=0.8pt] (8.,8.)-- (8.,12.);
\draw [line width=0.8pt] (8.,8.)-- (4.,6.);
\draw [line width=0.8pt] (8.,8.)-- (12.,6.);
\draw [line width=0.8pt,dash pattern=on 2pt off 2pt] (12.,6.)-- (20.,6.);
\draw [line width=0.8pt] (16.,12.)-- (12.,6.);
\draw [line width=0.8pt] (16.,12.)-- (20.,6.);
\draw [line width=0.8pt] (14.,4.)-- (20.,6.);
\draw [line width=0.8pt] (16.,12.)-- (14.,4.);
\draw [line width=0.8pt] (14.,4.)-- (12.,6.);
\draw [line width=0.8pt] (12.,6.)-- (16.,12.);
\draw (7.9,9) node[anchor=north west] {\begin{scriptsize}1\end{scriptsize}};
\draw (7.5,13.4) node[anchor=north west] {\begin{scriptsize}2\end{scriptsize}};
\draw (3.5,6.1) node[anchor=north west] {\begin{scriptsize}3\end{scriptsize}};
\draw (11.4,6.1) node[anchor=north west] {\begin{scriptsize}4\end{scriptsize}};
\draw (15.5,13.4) node[anchor=north west] {\begin{scriptsize}7\end{scriptsize}};
\draw (13.5,4.1) node[anchor=north west] {\begin{scriptsize}5\end{scriptsize}};
\draw (19.5,6.1) node[anchor=north west] {\begin{scriptsize}6\end{scriptsize}};
\begin{scriptsize}
\draw [fill=black] (8.,12.) circle (1.5pt);
\draw [fill=black] (4.,6.) circle (1.5pt);
\draw [fill=black] (12.,6.) circle (1.5pt);
\draw [fill=black] (8.,8.) circle (1.5pt);
\draw [fill=black] (20.,6.) circle (1.5pt);
\draw [fill=black] (16.,12.) circle (1.5pt);
\draw [fill=black] (14.,4.) circle (1.5pt);
\end{scriptsize}
\end{tikzpicture}
\end{center}
\caption{A $3$-uniform decomposable clutter}
\label{example of SC}
\end{figure}

We show that $\mathcal{C}$ is decomposable. Note that $\mathcal{C}=\mathcal{C}_1\cup\mathcal{C}_2$, where 
\[\mathcal{C}_1=\{\{1,2,3\},\{1,2,4\},\{1,3,4\}\}, \quad\quad\quad \mathcal{C}_2=\{\{4,5,6\},\{4,5,7\},\{4,6,7\},\{5,6,7\}\}.
\]
 Consider $\mathcal{C}_{4,3}$ on the vertex set $[4]$. Any maximal subcircuit of a complete clutter is a simplicial one, and so $\{2,3\}$ is simplicial in $\mathcal{C}_{4,3}$. Let $A_1=\{\{2,3,4\}\}$. Then $\C_1=\C_{4,3}\setminus A_1$. Hence $\C_1$ is a simplicial subclutter of $\C_{4,3}$. By definition, $\C_{4,3}$ is decomposable which follows  that $\C_1$ is also decomposable. On the other hand $\C_2=\C_{4,3}$ on the vertex set $\{4,5,6,7\}$. Hence $\C_2$ is also decomposable. Since $V(\C_1)\cap V(\C_2)=\{4\}$  is a clique in $\C_1$ and $\C_2$, we conclude that $\C$ is a decomposable clutter. 
\end{ex}

Let $\mathcal{C}=\mathcal{C}_1 \cup \mathcal{C}_2$ be a $d$-uniform non-complete clutter such that $V\left(\mathcal{C}_1\right)\cap V\left(\mathcal{C}_2\right)$ is a clique in $\mathcal{C}_1$ and $\mathcal{C}_2$. It is shown in \cite[Theorem 4.10 and Remark 4.12]{MYZ} that 
\begin{equation*}
\mathrm{reg}\left( I\left(\bar{\mathcal{C}} \right) \right) = \max \{ d, \mathrm{reg}\left( I\left(\bar{\mathcal{C}}_1 \right) \right), \mathrm{reg}\left( I\left(\bar{\mathcal{C}}_2 \right)\right)\}.
\end{equation*}
This, in particular, implies that the ideal $I\left(\bar{\mathcal{C}} \right)$ has a $d$-linear resolution if and only if both of the ideals $I\left(\bar{\mathcal{C}}_1 \right)$ and $I\left(\bar{\mathcal{C}}_2 \right)$ have linear resolutions. It is natural to ask whether the statement holds if we replace ``linear resolution" by ``linear quotients". Proposition~ \ref{clique} gives an affirmative answer to this question. To prove this proposition, we need the following lemma.  By $\mathrm{Mon}(S)$ we mean the set of all monomials in the polynomial ring $S$ and for $u\in \mathrm{Mon}(S)$, we let $\supp(u)=\{i\in[n]\colon \; x_i|u\}$. 

\begin{lem}\label{intermediate}
Let $I$ be a squarefree monomial ideal in $S=\mathbb{K}[x_1,\ldots,x_n]$, and let $N_1, N_2$ be disjoint subsets of $[n]$ such that 
$$\mathcal{G}(I)=\{u\in \mathrm{Mon}(S) \colon \  \deg u=d ,\ \supp(u)\cap N_1\neq \emptyset\neq \supp(u)\cap N_2\}.$$ 
Then $I$ has linear quotients. 
\end{lem}

\begin{proof}
Let $N_3$ be the set of indices $i\in [n]\setminus (N_1\cup N_2)$ such that $x_i$ divides $u$ for some $u\in \mathcal{G}(I)$. One may assume that $N_3=[n]\setminus (N_1\cup N_2)$.  We rename the elements of $[n]$ such that for all $i\in N_1$, $j\in N_2$, and $k\in N_3$ we have $i<j<k$.
Note that $I=\sum_{l=1}^{d-1}\sum_{l'=1}^lI_{\left(d-l\right)l'}$, where
\begin{align}\label{ajib}
I_{\left(d-l\right)l'}=\left(\prod_{t=1}^{d-l}x_{i_t}\prod_{s=1}^{l'}x_{j_s}\prod_{r=1}^{l-l'}x_{k_r}\colon \;  i_t\in N_1, j_s\in N_2, k_r\in N_3 \right).
\end{align}
We give the following order on the monomials in $\mathcal{G}(I)$:
\begin{itemize}
\item[(i)] for each $l,l'$ let  the monomial generators, $v_1,\ldots, v_r$,  of $I_{\left(d-l\right)l'}$ be ordered lexicographically induced by $x_1>x_2>\cdots>x_n$ such that $v_i<_{lex} v_{i+1}$ for $1\leq i<r$;
\item[(ii)] for  monomials $u\in \mathcal{G}\left(I_{\left(d-l_1\right)l'_1}\right)$ and $v\in \mathcal{G}\left(I_{\left(d-l_2\right)l'_2}\right)$ with $\left(l_1, l'_1\right) \neq \left(l_2, l'_2\right)$, let $u<v$ if and only if $\left(l_1, l'_1\right)< \left(l_2, l'_2\right)$; that is either (a)  $l_1<l_2$ or (b) $l_1=l_2$ and $l'_1<l'_2$;
\end{itemize}
Suppose that $I$ is generated by  monomials $u_1<\cdots<u_m$. We show that the colon ideal $\left(u_1, \ldots, u_{c-1}\right)\colon u_{c}$ is generated by variables, for $2\leq c\leq m$. 

Note that $\{u_i/ \gcd \left( u_i, u_{c} \right) \colon \ 1\leq i \leq c-1 \}$ is a set of generators of $\left(u_1, \ldots, u_{c-1}\right)\colon u_{c}$; see Proposition~\ref{colon ideal}. We show that for any  $u_i$, $i\leq c-1$, there exists $u_j$,  $j\leq c-1$, such that $u_j/\gcd\left(u_j,u_{c}\right)$ is of degree one and it divides   $u_i/ \gcd\left(u_i, u_{c}\right)$. 

By  (\ref{ajib}), $u_{c}$ belongs to $I_{\left(d-l_2\right)l'_2}$ for some $l_2, l'_2$. So $u_i$ belongs to $I_{\left(d-l_1\right)l'_1}$ for some $l_1, l'_1$, where $\left(l_1, l'_1\right) \leq\left(l_2, l'_2\right)$. Let 
\begin{align*}
u_i & = \prod_{t=1}^{d-l_1}x_{i'_t}\prod_{s=1}^{l_1'}x_{j'_s}\prod_{r=1}^{l_1-l_1'}x_{k'_r}, \text{ and}\\
u_{c} & = \prod_{t=1}^{d-l_2}x_{i_t}\prod_{s=1}^{l'_2}x_{j_s}\prod_{r=1}^{l_2-l'_2}x_{k_r}
\end{align*}
with $i_t, i'_t \in N_1$, $j_s,j'_s\in N_2$ and $k_r, k'_r\in N_3$. 
	
We consider different cases for $l_1,l_2,l'_1,l'_2$ and prove in each case that there exist $a,b\in [n]$ such that $x_a$ divides $u_{c}/u_i$ and $x_b$ divides $u_i/u_{c}$, and that $\left(u_{c}/x_a\right)x_b\in I$ with $\left(u_{c}/x_a\right)x_b<u_{c}$. Then, setting $u_j:=\left(u_{c}/x_a\right)x_b$, the assertion follows.
	
	\medspace
First suppose that $l_1=l_2$. Since $u_i<u_{c}$ we have $l'_1\leq l'_2$. Suppose that $l'_1=l'_2$. Then $ u_i\leq_{lex}u_{c}$. Since $u_i\neq u_{c}$, there exist $a,b\in [n]$ such that $x_a$ divides $u_{c}$ while it does not divide $u_i$, and $x_b$ divides $u_i$ while it does not divide $u_{c}$. Let $a,b$ be the smallest integers with these properties. Then $x_{b}<_{lex}x_a$. Therefore $\left(u_{c}/x_a\right)x_b\leq_{lex}u_{c}$. On the other hand since $l_1=l_2$ and $l'_1=l'_2$ and since $i<j<k$ for all $i\in N_1$, $j\in N_2$ and $k\in N_3$, it is easily seen that $a\in N_i$ if and only if $b\in N_i$. Thus $\left(u_{c}/x_a\right)x_b\in I_{\left(d-l_1\right)l'_1}\subseteq I$.
	
Assume now that $l'_1<l'_2$. It follows that there exists $a\in N_2$ such that $x_a$ divides $u_{c}$ but not $u_{i}$. In addition, $l_1-l'_1>l_2-l'_2$ implies that there exists $b\in N_3$ such that $x_b$ divides $u_{i}$ but not $u_{c}$. Then $\left(u_{c}/x_a\right)x_b\in I_{\left(d-l_2\right)\left(l'_2-1\right)}\subseteq I$. By the given ordering  we have $\left(u_{c}/x_a\right)x_b<u_{c}$.
	
Suppose that $l_1<l_2$. Then $d-l_1>d-l_2$ implies that there exists $b\in N_1$ such that $x_b$ divides $u_{i}$ and not $u_{c}$. In the case that $l'_1=l'_2$ we have $l_1-l'_1<l_2-l'_2$. It follows that there exists $a\in N_3$ such that $x_a$ divides $u_{c}$ and not $u_{i}$. Then $\left(u_{c}/x_a\right)x_b\in I_{\left(d-\left(l_2-1\right)\right)l'_2}\subseteq I$ and  $\left(u_{c}/x_a\right)x_b<u_{c}$. In the case that $l'_1<l'_2$ there exists $a\in N_2$ such that $x_a$ divides $u_{c}$ and it does not divide $u_{i}$. Then $\left(u_{c}/x_a\right)x_b\in I_{\left(d-\left(l_2-1\right)\right)\left(l'_2-1\right)}\subseteq I$ and $\left(u_{c}/x_a\right)x_b<u_{c}$. Finally, in the case that $l'_1>l'_2$ we have $l_1-l'_1<l_2-l'_2$. It follows that there exists $a\in N_3$ such that $x_a$ divides $u_{c}$ and it does not divide $u_{i}$. Then $\left(u_{c}/x_a\right)x_b\in I_{\left(d-\left(l_2-1\right)\right)l'_2}\subseteq I$ and $\left(u_{c}/x_a\right)x_b<u_{c}$.
\end{proof}

\begin{prop}\label{clique}
Let $\mathcal{C}$ be a $d$-uniform clutter on the vertex set $[n]$, $d\geq 2$, with $\mathcal{C}= \mathcal{C}_1 \cup \mathcal{C}_2$ such that $V\left(\mathcal{C}_1\right)\cap V\left(\mathcal{C}_2\right)$ is a clique in both $\mathcal{C}_1$ and $\mathcal{C}_2$. Then   $I\left(\bar{\mathcal{C}}\right)$ has linear quotients if and only if $I\left(\bar{\mathcal{C}}_1 \right)$ and  $I\left(\bar{\mathcal{C}}_2 \right)$ have linear quotients. 
 \end{prop}

\begin{proof}
Let $N_1=V\left(\mathcal{C}_1\right)\setminus V\left(\mathcal{C}_2\right)$,  $N_2=V\left(\mathcal{C}_2\right)\setminus V\left(\mathcal{C}_1\right)$, and $N_3=V\left(\mathcal{C}_1\right)\sect V\left(\mathcal{C}_2\right)$. Let $T=\{F\subseteq [n] \colon\; |F|=d, F\cap N_1\neq \varnothing \neq F\cap N_2\}$. Since $\mathcal{C}= \mathcal{C}_1 \cup \mathcal{C}_2$, we have $T\subseteq \bar{\mathcal{C}}$. Since $V\left(\mathcal{C}_1\right)\cap V\left(\mathcal{C}_2\right)$ is a clique in $\mathcal{C}_1$ and $\mathcal{C}_2$, for any $F\subseteq N_1\cup N_3$ with $F\notin \mathcal{C}_1$ we have $F\in \bar{\mathcal{C}}$. Similarly, for any $F\subseteq N_2\cup N_3$ with $F\notin \mathcal{C}_2$ we have $F\in \bar{\mathcal{C}}$. Thus $T\cup\bar{\mathcal{C}}_1 \cup \bar{\mathcal{C}}_2 \subseteq \bar{\mathcal{C}}$.
Conversely, suppose that $F\in \bar{\mathcal{C}}$. By assumption $F\not\subseteq N_3$. In the case that $F\subseteq N_1\cup N_3$ we have $F\in \bar{\mathcal{C}}_1$, and in the case that $F\subseteq N_2\cup N_3$ we have $F\in \bar{\mathcal{C}}_2$. Finally, in the case that $F\cap N_1\neq \varnothing \neq F\cap N_2$ we have $F\in T$. Consequently, $\bar{\mathcal{C}}=T\cup\bar{\mathcal{C}}_1 \cup \bar{\mathcal{C}}_2$. 
	
Since $V\left(\mathcal{C}_1\right)\cap V\left(\mathcal{C}_2\right)$ is a clique in $\mathcal{C}_1$ and $\mathcal{C}_2$ and $V\left(\bar{\mathcal{C}}_1 \right)=N_1\cup N_3$, $V\left(\bar{\mathcal{C}}_2 \right)=N_2\cup N_3$, we have $\mathcal{G}\left(I\left(\bar{\mathcal{C}}_1 \right)\right)\sect \mathcal{G}\left(I\left(\bar{\mathcal{C}}_2 \right)\right)=\varnothing$. 
Therefore $\mathcal{G}\left(I\left(\bar{\mathcal{C}}\right)\right)=\mathcal{G}\left(I\left(T\right)\right)\cup \mathcal{G}\left(I\left(\bar{\mathcal{C}}_1 \right)\right)\cup \mathcal{G}\left(I\left(\bar{\mathcal{C}}_2 \right)\right)$  is a disjoint union of the sets, where $I\left(T\right)$ is an ideal generated by monomials $\xb_F$  with $F\in T$. By Lemma~\ref{intermediate}, $I(T)$ has linear quotients. 

\medskip
First assume that $I\left(\bar{\mathcal{C}}_1 \right)$ and  $I\left(\bar{\mathcal{C}}_2 \right)$ have linear quotients. We show that    $I\left(\bar{\mathcal{C}}\right)$ has linear quotients.

We give the following  monomial ordering $<$ for the generators of $I:=I\left(\bar{\mathcal{C}}\right)$ such that the ideal $I$ has linear quotients with respect to this ordering:
\begin{itemize}
\item[(i)] for monomials $u\in \mathcal{G}\left(I(T)\right)$, consider the given order in Lemma~\ref{intermediate};
\item[(ii)] for  monomials $u\in \mathcal{G}\left(I(T)\right)$ and $v\in \mathcal{G}\left(I\left(\bar{\mathcal{C}}_1 \right)\right)\cup \mathcal{G}\left(I\left(\bar{\mathcal{C}}_2 \right)\right)$ let $u<v$; 
\item[(iii)] for  monomials $u\in \mathcal{G}\left(I\left(\bar{\mathcal{C}}_1 \right)\right)$ and $v\in \mathcal{G}\left(I\left(\bar{\mathcal{C}}_2 \right)\right)$ let $u<v$; 
\item[(iv)]  for  monomials $u\in \mathcal{G}\left(I\left(\bar{\mathcal{C}}_1 \right)\right)$ or $u\in \mathcal{G}\left(I\left(\bar{\mathcal{C}}_2 \right)\right)$, consider the given order of $I\left(\bar{\mathcal{C}}_1 \right)$ and $I\left(\bar{\mathcal{C}}_2 \right)$, by which these ideals have linear quotients.
\end{itemize}
	
Suppose that $I$ is generated by  monomials $u_1<\cdots<u_m$. We show that the colon ideal $\left(u_1, \ldots, u_{c-1}\right)\colon u_{c}$ is generated by variables, for   $2\leq c\leq m$. To do this, 
we prove that for any  $u_i$, $i\leq c-1$, there exists $u_j$,  $j\leq c-1$, such that $u_j/\gcd\left(u_j,u_{c}\right)$ is of degree one and it divides   $u_i/ \gcd\left(u_i, u_{c}\right)$. 

Suppose first that $u_{c}$ belongs to $I(T)$. So $u_i\in I(T)$. Hence by Lemma~\ref{intermediate} we get the desired result.

Suppose that $u_{c}\in I\left(\bar{\mathcal{C}}_1 \right)$. So $u_{c}=\xb_F$, where $F\subseteq N_1\cup N_3$. If $u_i \in I\left(\bar{\mathcal{C}}_1 \right)$ we are done, because $I\left(\bar{\mathcal{C}}_1 \right)$ has linear quotients. It is enough to show that $I\left(T\right)\colon u_{c}$ is generated by some variables. We claim that 
$$I\left(T\right)\colon u_{c}=\left(x_b\colon \; b \in N_2\right).$$
 Since $V\left(\mathcal{C}_1\right)\cap V\left(\mathcal{C}_2\right)$ is a clique in $\mathcal{C}_1$,  we have $F\cap N_1\neq \varnothing$.   In case $|F\cap N_1|\geq 2$, let $a\in F\cap N_1$, and in case $|F\cap N_1|=1$, since $|F|=d\geq 2$ and hence  $F\cap N_3\neq \varnothing$, let $a\in F\cap N_3$. 
 Therefore, for any $b\in N_2$,  $u_{j}:=\left(u_{c} /x_a\right)x_b \in I_{\left(d-l\right)1}\subseteq I\left(T\right)$ for some $l$, and so $u_{j}<u_{c}$. Moreover, $u_{j}/\gcd\left(u_{j}, u_{c}\right)=x_b$. So $\left(x_b\colon \; b\in N_2\right)\subseteq I\left(T\right) \colon u_{c}$. Conversely, for each $u\in \mathcal{G}\left(I\left(T\right)\right)$, there exists $x_b$ with $b\in N_2$ such that $x_b$ divides $u$. This implies that $x_b$ divides $u/\gcd\left(u,u_{c}\right)$, since $u_{c}=\xb_F$ with  $F\sect N_2=\varnothing$. Thus  $I\left(T\right) \colon u_{c}=\left(x_b\colon \; b\in N_2\right)$.
	
Finally, suppose that $u_{c}\in I\left(\bar{\mathcal{C}}_2 \right)$. So $u_{c}=\xb_F$, where $F\subseteq N_2\cup N_3$. If $u_i \in I\left(\bar{\mathcal{C}}_2 \right)$ we are done because $I\left(\bar{\mathcal{C}}_2 \right)$ has linear quotients. It is enough to show that the ideal $\left(I\left(T\right)+I\left(\bar{\mathcal{C}}_1 \right)\right)\colon u_{c}$ is generated by some  variables. We claim that 
$$\left(I\left(T\right)+I\left(\bar{\mathcal{C}}_1 \right)\right) \colon u_{c}=\left(x_b\colon \; b\in N_1\right).$$
Similar to the above argument we have   $\left(x_b\colon \; b\in N_1\right)\subseteq \left(I\left(T\right)+I\left(\bar{\mathcal{C}}_1 \right)\right) \colon u_{c}$. Conversely, for each $u\in \mathcal{G}\left(I\left(T\right)+I\left(\bar{\mathcal{C}}_1 \right)\right)$, $u=\xb_G$ with $G\sect N_1\neq \varnothing$, because $N_3$ is a clique in $\mathcal{C}_1$. Thus there exists $x_b$ with $b\in N_1$ such that $x_b$ divides $u$. This implies that $x_b$ divides $u/\gcd\left(u,u_{c}\right)$, because $u_{c}=\xb_F$ with  $F\sect N_1=\varnothing$. It follows that $\left(I\left(T\right)+I\left(\bar{\mathcal{C}}_1 \right)\right) \colon u_{c}=\left(x_b \colon \; b\in N_1\right)$. This completes the proof that $I(\bar{\mathcal{C}})$ has linear quotients.

Now suppose   $I\left(\bar{\mathcal{C}}\right)$ has linear quotients. We prove that the both ideals $I\left(\bar{\mathcal{C}}_1 \right)$ and  $I\left(\bar{\mathcal{C}}_2 \right)$ have linear quotients. 
Suppose $I$ is minimally generated by $\mathbf{x}_{F_1},\ldots,\mathbf{x}_{F_m}$ and suppose the given order provides linear quotients for $I$. 
Let $\mathcal{G}\left(I\left(\bar{\mathcal{C}_1}\right)\right)=\{\mathbf{x}_{F_{i_1}},\ldots,\mathbf{x}_{F_{i_l}}\}$ with $1\leq i_1<\cdots<i_l\leq m$. We show that $I(\bar{\C_1})$ has linear quotients with the given order of the generators. 
To do this, we prove that the colon ideal $J_{\C_1}:=(\mathbf{x}_{F_{i_1}},\ldots,\mathbf{x}_{F_{i_{k-1}}})\colon \mathbf{x}_{F_{i_k}}$  is generated by variables  for all $i_k$ with $i_1<i_k\leq i_l$. 
Consider a (not necessarily minimal) generator  $\mathbf{x}_{F_{i_s}}/\gcd(\mathbf{x}_{F_{i_s}},\mathbf{x}_{F_{i_k}})$ of $J_{\C_1}$. Since $\mathbf{x}_{F_{i_s}}/\gcd(\mathbf{x}_{F_{i_s}},\mathbf{x}_{F_{i_k}})$ belongs to the colon  ideal $J_{\C}=(\mathbf{x}_{F_{1}}, \mathbf{x}_{F_{2}}, \ldots,\mathbf{x}_{F_{i_{k}-1}}):\mathbf{x}_{F_{i_k}}$ and since by assumption $J_{\C}$ is generated by variables, it follows that there exists $1\leq r<i_k$ such that $\mathbf{x}_{F_{r}}/\gcd(\mathbf{x}_{F_{r}},\mathbf{x}_{F_{i_k}})$ is of degree $1$ and it divides $\mathbf{x}_{F_{i_s}}/\gcd(\mathbf{x}_{F_{i_s}},\mathbf{x}_{F_{i_k}})$. We show that $\mathbf{x}_{F_{r}} \in \mathcal{G}\left(I\left(\bar{\mathcal{C}_1}\right)\right)$ which proves the assertion.  It follows from  $$\frac{\mathbf{x}_{F_{r}}}{\gcd(\mathbf{x}_{F_{r}},\mathbf{x}_{F_{i_k}})}\big| \frac{\mathbf{x}_{F_{i_s}}}{\gcd(\mathbf{x}_{F_{i_s}},\mathbf{x}_{F_{i_k}})}$$
that $F_r\setminus F_{i_k}\subseteq F_{i_s}\setminus F_{i_k}$. 
Since $F_{i_s},  F_{i_k}$ belong to $\bar{\C_1}$ we have $F_{i_s}, F_{i_k}\subseteq V(\C_1)$. Therefore $F_r\setminus F_{i_k}, F_r\cap F_{i_k}\subseteq V(\C_1)$. This implies that $F_r\subseteq V(\C_1)$. Hence $F_r\notin T$.  If $F_r\subseteq V(\C_2)$, then since $V(\C_1)\cap V(\C_2)$ is a clique in $\C_1$, $\C_2$, we have $F_r\in \C_1\cap \C_2$ and so $\mathbf{x}_{F_r}\notin I(\bar{\C})$, a contradiction. Thus $F_r\not\subseteq V(\C_2)$. It follows that 
$\mathbf{x}_{F_r}\notin  \mathcal{G}\left(I\left(T\right)\right)\cup \mathcal{G}\left(I\left(\bar{\mathcal{C}_2}\right)\right)$. Consequently, $\mathbf{x}_{F_r}\in  \mathcal{G}\left(I\left(\bar{\mathcal{C}_1}\right)\right)$.
\end{proof}

Now we have all tools needed to prove the main theorem of this section:
\begin{thm}\label{stronglinq}
Let $\mathcal{C}$ be a decomposable clutter on the vertex set $[n]$. Then the ideal $I\left(\bar{\mathcal{C}}\right)$ has linear quotients.
\end{thm}

\begin{proof}
If  $\mathcal{C}=\mathcal{C}_{n,d}$, then $I(\bar{\mathcal{C}})=0$ and there is nothing to prove. So assume that $\mathcal{C} \neq \mathcal{C}_{n,d}$. If $\mathcal{C}$ is  of the form $\mathcal{C}= \mathcal{C}_1 \cup \mathcal{C}_{2}$, where $\mathcal{C}_1, \mathcal{C}_2$ are decomposable clutters on the vertex set $V(\mathcal{C}_1)$, $V(\mathcal{C}_2)$, respectively, with $V(\mathcal{C}_1) \nsubseteq V(\mathcal{C}_2)$, $V(\mathcal{C}_2) \nsubseteq V(\mathcal{C}_1)$ and $V\left(\mathcal{C}_1 \right)\cap V\left(\mathcal{C}_2 \right)$ is a clique in both $\mathcal{C}_1$ and $\mathcal{C}_2$, then the ideal $I\left(\bar{\mathcal{C}}\right)$ has linear quotients by using induction on the number of vertices and Proposition~\ref{clique}. Suppose now that $\mathcal{C}$ is  not of this form. Then $\mathcal{C}$ is a simplicial subclutter of a decomposable $d$-uniform clutter $\C'$. By induction on $|\mathcal{G}(I(\bar{\C}))|$, we conclude that the ideal $I(\bar{\C}')$ has linear quotients. Applying Theorem~\ref{simplicial subclutter has linear quotients-cor}(ii), we conclude that the ideal $I\left(\bar{\mathcal{C}}\right)$ has linear quotients.
\end{proof}

As a consequence of Theorem~\ref{stronglinq} we have the following result.

\begin{cor}\label{shellability of simp. of decomposables}
Let $\mathcal{C}$ be a decomposable clutter. Then $\Delta(\mathcal{C})^\vee$ is shellable.
\end{cor}

\begin{proof}
By Theorem~\ref{stronglinq} the ideal $I_{\Delta(\mathcal{C}) }= I( \bar{\mathcal{C}} )$ has  linear quotients. It follows from Proposition~\ref{linear quotient and shellability} that $\Delta(\mathcal{C})^\vee$ is shellable.
\end{proof}

Example~\ref{lin quo not decomposable} shows that the class of decomposable clutters is not equivalent to the class of uniform clutters whose associated ideals have linear quotients. Note that for a uniform clutter $\mathcal{C}$ and  a simplicial element $e$ of $\mathcal{C}$, if $A\subseteq \{F\in \mathcal{C}\colon e\subset F\}$, then $e\in\mathrm{Simp}(\mathcal{C}\setminus A)$.
\begin{ex}\label{lin quo not decomposable}
Let 
$$
{\mathcal C} =  \{ \{1,2,5\},\{1,3, 5\},\{1,4,5\},\{2, 3, 4\}\}
$$
be a $3$-uniform clutter on $[5]$. 
It is easy to check that the ideal $I(\bar{\mathcal{C}})$ has linear quotients with respect to the following order:
$$
I(\bar{\mathcal C})=(x_1x_2x_3, x_1x_2x_4, x_1x_3x_4, x_3x_4x_5, x_2x_4x_5, x_2x_3x_5). 
$$
Suppose $\mathcal{C}$ is a simplicial subclutter of a decomposable clutter ${\mathcal D}$. Since $\mathcal{C}\neq {\mathcal D}$, there exists a simplicial sequence $e_1, \ldots, e_t$ of ${\mathcal D}$ and  non-empty subsets $A_1,\ldots, A_t$ as defined in Subsection~\ref{simp sub}, such that $\mathcal{C}={\mathcal D}\setminus A_1\setminus\cdots\setminus A_t$.  Then for any  $F\in  A_t$ we have $e_t\in\mathrm{Simp}(\mathcal{C}\cup\{F\})$. Note that $e_t\neq \{1,5\}$ because $\{1,5\}\notin\mathrm{Simp}(\mathcal{C})$. 

Since all $2$-subsets of $[5]$ are maximal subcircuits of $\mathcal{C}$, by symmetry, we may assume that either $e_t=\{1,2\}$ or $e_t=\{2,3\}$. If $e_t=\{1,2\}$, then $A_t\subseteq\{\{1,2,3\},\{1,2,4\}\}$. But $\{1,2\}$ is not simplicial in either of $\mathcal{C}\cup\{\{1,2,3\}\}$ or  $\mathcal{C}\cup\{\{1,2,4\}\}$. Therefore $A_t=\varnothing$, a contradiction. Hence $e_t\neq \{1,2\}$. With the similar argument, we also see that $e_t\neq \{2,3\}$. It follows that $\mathcal{C}$ is not a simplicial subclutter of $\mathcal{D}$. 

Now, assume that  ${\mathcal C}={\mathcal C_1} \cup {\mathcal C_2}$ such that ${\mathcal C_1}$ and ${\mathcal C_2}$ are decomposable clutters and $V(\mathcal{C}_1)\cap V(\mathcal{C}_2)$ is a clique in both  ${\mathcal C_1}$ and ${\mathcal C_2}$.  Let $i\in V({\mathcal C_1})\setminus V({\mathcal C_2})$ and $j\in V({\mathcal C_2})\setminus V({\mathcal C_1})$. Then there is no circuit in $\mathcal{C}$ containing $\{i,j\}$. But,  all $2$-subsets of $[5]$ are maximal subcircuits of $\mathcal{C}$, a contradiction. Therefore ${\mathcal C}\neq{\mathcal C_1} \cup {\mathcal C_2}$. Consequently $\mathcal{C}$ is not a decomposable clutter. 
\end{ex}


\section{Simon's Conjecture and chordality of decomposable clutters} \label{Simon}

In this section we study the relation between  extendable shellability  of  skeletons of a simplex and chordality of decomposable clutters.

\begin{defn}[{see \cite{Klee}}] \label{Definition of decomposable}
A Simplicial complex $\Delta$ is called \textit{extendably shellable}, if any shelling
of a subcomplex $\Gamma$ of $\Delta$ with $\mathcal{F}(\Gamma)\subseteq \mathcal{F}(\Delta)$ can be extended to a shelling of $\Delta$.
\end{defn}

As mentioned in the introduction, few classes of extendably shellable simplicial complexes are known. The following conjecture by R. S. Simon offers a big class of extendably shellable complexes.

\begin{conj}[{Simon's Conjecture, \cite[Conjecture~4.2.1]{Simon}}] \label{Simon's conjecture}
Every $d$-skeleton of a simplex is extendably shellable.
\end{conj}

In case $d = 2$  this conjecture was shown to be true by  Bj\"orner and Eriksson \cite{Bjorner}. In Corollary~\ref{partial answer to Simon's conjecture}, we  show that  Conjecture~\ref{Simon's conjecture} also holds   in case $d \geq n-3$. The main tool to prove this result  is Proposition~\ref{Simon implies that complete clutter is chordal}. We need the following lemma for the proof of this proposition.

\begin{lem}\cite[Remark~2.2(ii)]{NZ} \label{linear quotients and simplicial order}
Let $\mathcal{C}_{n,d}$ be the complete $d$-uniform clutter on the vertex set $[n]$ and let $\eb= e_1, e_2, \ldots, e_t$ be a sequence of $(d-1)$-subsets of $[n]$. The followings are equivalent:
\begin{itemize}
\item[\rm (i)] the sequence $\eb= e_1, e_2, \ldots, e_t$ is a simplicial sequence in $\mathcal{C}_{n,d}$;
\item[\rm (ii)] the ideal $I=\left(\xb_{e_1}, \ldots, \xb_{e_t}\right)$ has linear quotients,  with respect to the order $\xb_{e_1}, \ldots, \xb_{e_t}$.
\end{itemize}
\end{lem}

\begin{rem}\label{changing the position in lin quo}
Let $\mathcal{C}$ be a $d$-uniform clutter on $[n]$ and let $e$ be a $(d-1)$-subset of $[n]$ which is not a maximal subcircuit of  $\mathcal{C}$.  By definition of a simplicial element, we know that $e$ is simplicial over $\mathcal{C}$. However, $\mathcal{C}\setminus e=\mathcal{C}$.  Hence,  a simplicial sequence ${\bf e}=e_1,\ldots, e_t$ of the complete clutter $\mathcal{C}_{n,d}$ may contain some $e_i$ such that by removing $e_i$ from ${\bf e}$ we get a new simplicial sequence ${\bf e'}$ of $\mathcal{C}_{n,d}$  of length $t-1$ with  $(\mathcal{C}_{n,d})_{{\bf e'}}^{t-1}=(\mathcal{C}_{n,d})_{{\bf e}}^t$. Similarly, if if ${\bf e}=e_1,\ldots,e_t$ is a simplicial sequence of $\mathcal{C}_{n,d}$ and $e_{t+1}$ is a $(d-1)$-subset of $[n]$ with $e_{t+1}\notin \mathrm{SC}\left( (\mathcal{C}_{n,d})_{\bf e}^t \right)$, then adding $e_{t+1}$ to the end of ${\bf e}$ we get a simplicial sequence ${\bf e'}$  of $\mathcal{C}_{n,d}$ of length $t+1$ with  $(\mathcal{C}_{n,d})_{{\bf e'}}^{t+1}=(\mathcal{C}_{n,d})_{{\bf e}}^t$.

It is easy to check that a $(d-1)$-subset $e$ of $[n]$ is not a maximal subcircuit of $(\mathcal{C}_{n,d})_{{\bf e}}^t$ if and only if $\left(\xb_{e_1}, \ldots, \xb_{e_t}\right) \colon  \xb_{e}=(x_i\colon i\in [n]\setminus e)$. Applying this fact to Lemma~\ref{linear quotients and simplicial order}, one observes that given an equigenerated squarefree monomial ideal $I=\left(\xb_{e_1}, \ldots, \xb_{e_t}\right)$ which has linear quotients with respect to the given order, if  $\left(\xb_{e_1}, \ldots, \xb_{e_{j-1}}\right) \colon  \xb_{e_j}=(x_i\colon i\in [n]\setminus e_j)$ for some $j$, then   removing $ \xb_{e_j}$ form the ideal or changing its position to the $k$-th position with $k\geq j$ in the given order,   the resulting ideal with the new order still has linear quotients. 

It follows from \cite[Proposition~4.6]{BHYZ} that all  simplicial orders of $\mathcal{C}_{n,d}$ are of length  $l=\Sigma_{i=1}^{n-d+1} {{n-1-i}\choose{d-2}} = \Sigma_{i=0}^{n-2} {i \choose {d-2}}$. Therefore, if a squarefree monomial ideal $I$ generated in degree $d-1$ has at least $l$ minimal generators, the ideal $I$ has linear quotients if and only if there exist $l$ elements $u_1, \ldots , u_l$ in the minimal generating set of $I$ with the property that $J=(u_1, \ldots , u_l)$ has linear quotients with the given order and there is no minimal generator $u_j$ in $J$ such that $(u_1,\ldots, u_{j-1}) \colon u_j=(x_i\colon i\in [n]\setminus \supp(u))$. Then the ideal $I$ has linear quotients with respect to the order $u_1,\ldots, u_l, u_{l+1}, \ldots, u_r$, where $u_{l+1}, \ldots, u_r$ are the minimal generators of $I$ not belonging to $J$. 
\end{rem}

\begin{prop} \label{Simon implies that complete clutter is chordal}
Let $\mathcal{C}_{n,d}$ be the complete $d$-uniform clutter on the vertex set $[n]$. The followings are equivalent:
\begin{itemize}[topsep=8pt, partopsep=5pt]
\item[\rm (i)] The simplicial complex $\langle [n] \rangle^{(n-d)}$ is extendably shellable;
\item[\rm (ii)] If $\textbf{\rm{\textbf{e}}}= e_1, e_2, \ldots, e_t$ is a simplicial sequence in $\mathcal{C}_{n,d}$, then the clutter $\mathcal{C} = \left( \mathcal{C}_{n,d} \right)_{\textbf{\rm{\textbf{e}}}}^{t}$ is a chordal clutter;
\end{itemize}
\end{prop}

\begin{proof}
(i)$\Rightarrow$(ii): If $\mathcal{C} =\varnothing$, there is nothing to prove. Assume that $\mathcal{C} \neq \varnothing$.
Let $\Delta= \langle [n] \rangle^{(n-d)}$ and consider the subcomplex $\Delta' = \langle {[n]\setminus e_1}, \ldots, {[n]\setminus e_t} \rangle$ of $\Delta$. Note that $\Delta' \neq \Delta$, because $\mathcal{C} \neq \varnothing$.
Since $\eb= e_1, e_2, \ldots, e_t$ is a simplicial sequence in $\mathcal{C}_{n,d}$, it follows from Proposition~\ref{linear quotient and shellability} and Lemma~\ref{linear quotients and simplicial order} that $\Delta'$ is shellable with the given order. Since $\Delta$ is extendably shellable, there exist $G_{t+1}, \ldots, G_r \in \mathcal{F} \left( \Delta \right)$, $r={n \choose n-d+1}$, such that $\Delta=\langle {[n] \setminus e_1}, \ldots, {[n] \setminus e_t}, G_{t+1}, \ldots, G_r \rangle$ is shellable with the given order. For $i=t+1, \ldots, r$ let $e_{i} = [n] \setminus G_i$. Then by Proposition~\ref{linear quotient and shellability}, the ideal $\left(\xb_{e_1}, \ldots, \xb_{e_r} \right)$ has linear quotients, with the given order. Using Lemma~\ref{linear quotients and simplicial order}  once again, we conclude that ${e_1}, \ldots, {e_r}$ is a simplicial sequence in $\mathcal{C}_{n,d}$. Note that $\mathcal{C}\setminus e_{t+1}\setminus \cdots\setminus e_r=\mathcal{C}_{n,d}\setminus e_1\setminus \cdots\setminus e_r =\varnothing$, for $r= {n \choose d-1}$. This implies that $\mathcal{C}$ is chordal.

(ii)$\Rightarrow$(i): Let $\Delta'=\langle G_1,\ldots, G_t\rangle$ be a subcomplex of $\langle [n] \rangle^{(n-d)}$ with $\dim G_i=n-d$ for all $i$, and suppose that $\Delta'$ is shellable with the given order of the facets. For each $i$, let $e_i=[n]\setminus G_i$. It follows from Proposition~\ref{linear quotient and shellability} that the ideal $(\xb_{e_1}, \ldots, \xb_{e_t})$ has linear quotients with the given order of the generators, and Lemma~\ref{linear quotients and simplicial order} implies that ${e_1}, \ldots, {e_t}$ is a simplicial sequence in $\mathcal{C}_{n,d}$. By assumption, the clutter  $\mathcal{C} = \left( \mathcal{C}_{n,d} \right)_{\textbf{\rm{\textbf{e}}}}^{t}$ is  chordal. Therefore there exists a simplicial sequence $\mathbf{e}=e_{t+1},\ldots, e_{s}$  of $\mathcal{C}$  such that $\mathcal{C}\setminus e_{t+1}\setminus\cdots\setminus e_s=\varnothing$. It follows that $e_1,\ldots,e_s$ is a simplicial sequence for $\mathcal{C}_{n,d}$ with $\mathcal{C}_{n,d}\setminus e_1\setminus \cdots\setminus e_s=\varnothing$. Let $\{e_{s+1},\ldots, e_r\}=\{e\subseteq [n]:\ |e|=d-1,\ e\neq e_i,\ 1\leq i\leq s\}$. Then for each $s+1\leq i\leq r$, $e_i$ is not a maximal subcircuit of $\mathcal{C}_{n,d}\setminus e_1\setminus \cdots\setminus e_s$ and hence by Remark~\ref{changing the position in lin quo}  ${\bf e'}=e_1,\ldots, e_r$ is a simplicial sequence of $\mathcal{C}_{n,d}$. 
By Lemma~\ref{linear quotients and simplicial order} the ideal $(\xb_{e_1}, \ldots, \xb_{e_r})$ has linear quotients with the given order of the generators.  Proposition~\ref{linear quotient and shellability} implies that $\langle [n]\rangle^{(n-d)}=\langle G_1,\ldots, G_t,[n]\setminus e_{t+1},\ldots, [n]\setminus e_r\rangle$ is  shellable with respect to the given order. This completes the proof.
\end{proof}

Let $G$ be a decomposable graph. Then by Theorem~\ref{stronglinq} the ideal $I(\bar{G})$ has linear quotients. Hence it has linear resolution, and so by Theorem~\ref{Froberg}, $G$ is chordal. Hence we have $\mathfrak{C}'_2 \subseteq \mathfrak{C}_2$ (indeed, by Dirac's celebrated theorem \cite{Dirac}, we have $\mathfrak{C}'_2 =\mathfrak{C}_2$). This fact together with  Proposition~\ref{pure skeleton of quasi-forest}, Corollary~\ref{stable,lex} and some experimental evidence, lead us to the following conjecture:

\begin{conj} \label{decomposables are chordal-conjecture}
$\mathfrak{C}'_d \subset \mathfrak{C}_d$, for all $d$.
\end{conj}
Next corollary shows that Conjecture~\ref{decomposables are chordal-conjecture} is, indeed, a generalization of Simon's conjecture.

\begin{cor} \label{Inverse of Simon's conjecture}
If $\mathfrak{C}'_d \subset \mathfrak{C}_d$, then $\langle [n] \rangle^{(n-d)}$ is  extendably shellable.
\end{cor}

\begin{proof}
It follows from Definition~\ref{Definition of decomposable} that $(\mathcal{C}_{n,d})_\textbf{e}^t$ is a decomposable clutter for any choice of a simplicial sequence $\textbf{e} = e_1, \ldots, e_t$. Also, our assumption implies that $(\mathcal{C}_{n,d})_\textbf{e}^t$ is a chordal clutter. In view of Proposition~\ref{Simon implies that complete clutter is chordal}, we get the desired conclusion.
\end{proof}

\begin{cor} \label{partial answer to Simon's conjecture}
\begin{itemize}
\item[\rm (i)] For $d=2,3$ the clutter $(\mathcal{C}_{n,d})_{\textbf{e}}^t$ is chordal  for any simplicial sequence $\textbf{e}=e_1,\ldots, e_t$ of $\mathcal{C}_{n,d}$.
\item[\rm (ii)] For $i \geq n-3$, the simplicial complex $\langle [n] \rangle^{\left( i \right)}$ is extendably shellable. 
\end{itemize}
\end{cor}

\begin{proof}
(i) Let $\mathcal{C}= (\mathcal{C}_{n,d})_{\textbf{e}}^t$. If $\mathcal{C}=\varnothing$ we are done. Suppose $\mathcal{C}\neq\varnothing$.

First suppose $d=2$. 
In this case, $\mathcal{C}_{n,2}$ is the complete graph on the vertex set $[n]$ and a simplicial deletion with respect to a simplicial vertex $v$  results in the complete graph $\mathcal{C}_{n-1, d}$ on the vertex set $[n]\setminus \{v\}$.  Hence   $\mathcal{C}= \mathcal{C}_{n-t, 2}$ is a complete graph on $n-t$ vertices. This graph is clearly a chordal graph. 

Now suppose $d=3$. 
Since $e_1, \ldots, e_t$ is a simplicial sequence in $\mathcal{C}_{n,3}$, it follows from Lemma~\ref{linear quotients and simplicial order} that the ideal $I=\left( \xb_{e_1}, \ldots, \xb_{e_t} \right)$, which is equigenerated in degree $2$, has linear quotients, and hence linear resolution. Let $G$ be the graph with $I \left( \bar{G} \right) = I$. Then $G$ is a chordal graph, see Theorem~\ref{Froberg}. Thus $G$ admits a simplicial order ${\bf e}$. Suppose $v$ is a vertex which is the first element of the sequence ${\bf e}$ and let $w \in \mathrm{N}_G \left[v \right]\setminus\{v\}$. Now let $G'$ be a graph whose edges are all edges of $G$ except $\{v,w\}$. Then  $G'$ is again a chordal graph and since $G'$ is a simplicial subclutter of $G$, it follows from Lemma~\ref{simplicial subclutter has linear quotients-cor} that the ideal $I \left( \bar{G'} \right) = \left( \xb_{e_1}, \ldots, \xb_{e_t}, x_vx_w \right)$ has linear quotients with the given order.  Lemma~\ref{linear quotients and simplicial order}  implies that $e_1, \ldots, e_t, e_{t+1}$ is a simplicial sequence in $\mathcal{C}_{n,3}$, where $e_{t+1} = \{v, w\}$. Continuing this process for $\mathcal{C} \setminus e_{t+1}$,  after some finite steps, we  find a simplicial sequence $e_{t+1}, \ldots, e_r$ in $\mathcal{C}$, such that $\mathcal{C}\setminus e_{t+1}\setminus \cdots\setminus e_r= \varnothing$. Therefore $\mathcal{C}$ is a chordal clutter.

\medskip
(ii) The assertion is clear for $i \geq n-1$. For  $i=n-3, n-2$, the assertion follows from  Proposition~\ref{Simon implies that complete clutter is chordal} and part (i).
\end{proof}


\section{Some classes of decomposable clutters}
The aim of this section is to compare the class of ideals associated to  decomposable clutters with some other known classes of ideals with linear quotients. Indeed, since the ideals associated to decomposable (chordal, resp.) clutters have linear quotients (resolution, resp.), it is reasonable to ask how large this class is. We will see that some classes of ideals with linear quotients come from the class of decomposable clutters.


\subsection{Quasi-forest simplicial complexes}
Let $\Delta$ be a simplicial complex. A facet $F \in  \mathcal{F}\left(\Delta \right)$ is said to be a \textit{leaf} of $\Delta$ if either $F$ is the only facet of $\Delta$, or there exists a facet $G\in \mathcal{F}\left(\Delta \right)$ with $G \neq F$, called a \textit{branch} of $F$, such that $H\cap F\subseteq G\cap F$ for all $H \in \mathcal{F}\left(\Delta \right)$ with $H \neq F$.

A vertex $v$ of $\Delta$ is called a \textit{free vertex }of $\Delta$ if $v$ belongs to exactly one facet. Note that every leaf has at least one free vertex.

A \textit{quasi-forest} is a simplicial complex such that there exists a labeling $F_1, \ldots, F_q$ of the facets of $\Delta$, called a \textit{leaf order}, such that for each $1< i \leq q$ the facet $F_i$ is a leaf of the subcomplex $\langle F_1, \ldots, F_i\rangle$.  

\medskip
It is known that the $1$-skeleton of a quasi-forest is  a chordal graph (essentially Dirac \cite{Dirac}, see \cite[Theorem 3.3]{HHZ}). In the following, we show that  every pure $d$-skeleton of a quasi-forest is a decomposable chordal clutter.

\begin{prop} \label{pure skeleton of quasi-forest}
Let $\Delta$ be a quasi-forest, $d$ a positive integer and let $\mathcal{C} = \mathcal{F} \left( \Delta^{[d]} \right)$. Then $\mathcal{C}$  is a decomposable and chordal clutter.
\end{prop}

\begin{proof}
Let $[n]$ be the vertex set of $\Delta$ and $\mathcal{F} \left( \Delta \right) = \left\{ F_1, \ldots, F_r \right\}$, where $F_i$ is a leaf in the simplicial complex $\langle F_1, \ldots, F_i\rangle$ for $i=2, \ldots, r$. 
We use induction on $r$ to show that $\mathcal{C} = \mathcal{F} \left( \Delta^{[d]} \right)$ is a decomposable chordal clutter.

 If $r=1$, then $\mathcal{C}=\mathcal{F}(\langle F_1 \rangle^{[d]})$ consists of all $(d+1)$-subsets of $F_1$ and hence it is a complete $(d+1)$-uniform clutter on the vertex set $F_1$. Therefore it is decomposable by definition.  Moreover by \cite[Corollary 3.11]{BYZ}, $\mathcal{C}$ is  also chordal.  Assume that $r>1$ and the assertion holds for all quasi-forest simplicial complexes with less facets. Let $F$ be the set of free vertices of $F_r$ in $\Delta$, let $F'=F_r \setminus F$ and $\mathcal{D} = \mathcal{C} \lceil_{[n] \setminus F}$. 
 Since $F_r$ is a leaf in $\Delta$, there exits $j<r$, such that $F_i \cap F_r \subseteq F_j \cap F_r$, for all $i=1, \ldots, r-1$. Hence 
\begin{equation*}
F' = \mathop{\bigcup}_{i=1}^{r-1} \left( F_i \cap F_r \right) \subseteq F_j \cap F_r \subseteq F_j.
\end{equation*}
This implies that $F' \in \langle F_1, \ldots, F_{r-1} \rangle$. Set $\mathcal{C}_1 = \mathcal{F} \left( \langle F_1, \ldots, F_{r-1} \rangle^{[d]} \right)$ and $\mathcal{C}_2= \mathcal{F}(\langle F_r \rangle^{[d]})$. Then  $\mathcal{C} = \mathcal{C}_1 \cup \mathcal{C}_2$, and $V \left( \mathcal{C}_1 \right) \cap F_r = F'$ is a clique in $\mathcal{C}_1$ and $\mathcal{C}_2$. 

Since the simplicial complex $\langle F_1, \ldots, F_{r-1} \rangle$ is a quasi-forest, our induction hypothesis implies that $\mathcal{C}_1$ is a decomposable chordal clutter.  Moreover, $\mathcal{C}_2$ is also decomposable and chordal by induction base. Therefore, by definition,   $\mathcal{C}$ is decomposable. 

As in the proof of \cite[Lemma 3.10]{BYZ}, there exists a simplicial sequence $\textbf{e}=e_1, \ldots, e_t$ in $\mathcal{C}$ such that  $\mathcal{C}_\textbf{e}^t = \mathcal{C}_1$. 
It follows that $\mathcal{C}$ is also a chordal clutter.
\end{proof}

\begin{rem} \mbox{}
\begin{itemize}
\item[(a)] Let $\Delta$ be a quasi-forest simplicial complex, $\mathcal{C} = \mathcal{F} \left( \Delta^{[d]} \right)$ and $I= I \left( \bar{\mathcal{C}} \right)$.  It follows from Proposition~\ref{pure skeleton of quasi-forest} and Theorem~\ref{stronglinq} that the ideal $I$ has linear quotients. 
\item[(b)] It is known that every chordal graph is  $1$-skeleton of a quasi-forest (c.f. \cite[Theorem 3.3]{HHZ}). Since chordal graphs are exactly decomposable $2$-uniform clutters, we conclude  that every decomposable graph is  $1$-skeleton of a quasi-forest. This is not the case for arbitrary $d$-uniform decomposable clutters, $d\geq 3$; see Example~\ref{ex 1}. 
\item[(c)] Let $G$ be a decomposable graph and $I= I \left(\bar{G} \right)$. Then $G$ is chordal and it is known that all powers of $I$ have a linear resolution (c.f. \cite[Theorem 3.2]{HHZ2}). One can not expect to have the same statement for arbitrary $d$-uniform decomposable clutters, $d\geq 3$; see Example~\ref{ex 2}. 
\end{itemize}
\end{rem}

\begin{ex} \label{ex 1}
Let $\mathcal{C}$ be the following $3$-uniform clutter (see Figure~\ref{umbrella}):
\begin{equation*}
\mathcal{C} = \left\{ \{1,2,3\}, \{1,3,4\}, \{1,4,5\}, \{1,2,5\} \right\}.
\end{equation*}

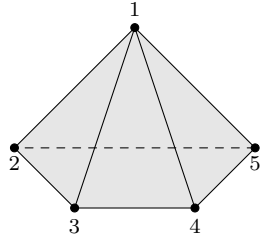
\begin{figure}[H]
\centering
\begin{tikzpicture}[line cap=round,line join=round,>=triangle 45,x=.8cm,y=.8cm]
\clip(0.4,0.5) rectangle (8.6,4.5);
\fill[line width=1.2pt,fill=black,fill opacity=0.10000000149011612] (5.,4.) -- (3.,2.) -- (4.,1.) -- cycle;
\fill[line width=1.2pt,fill=black,fill opacity=0.10000000149011612](5.,4.) -- (6.,1.) -- (4.,1.) -- cycle;
\fill[line width=1.2pt,fill=black,fill opacity=0.10000000149011612](5.,4.) -- (7.,2.) -- (6.,1.) -- cycle;
\draw (4.,1.)-- (6.,1.);
\draw [dash pattern=on 3pt off 3pt]  (7.,2.)-- (3.,2.);
\draw (3.,2.)-- (4.,1.);
\draw (7.,2.)-- (6.,1.);
\draw (5.,4.)-- (3.,2.);
\draw (5.,4.)-- (4.,1.);
\draw (5.,4.)-- (6.,1.);
\draw (5.,4.)-- (7.,2.);
\begin{scriptsize}
\draw (4.0,0.7) node {\begin{scriptsize}$3$\end{scriptsize}};
\draw (6.0,0.7) node {\begin{scriptsize}$4$\end{scriptsize}};
\draw (7.0,1.75) node {\begin{scriptsize}$5$\end{scriptsize}};
\draw (3.0,1.75) node {\begin{scriptsize}$2$\end{scriptsize}};
\draw (5.0,4.3) node {\begin{scriptsize}$1$\end{scriptsize}};
\end{scriptsize}
\draw [fill=black] (5.,4.) circle (1.5pt);
\draw [fill=black] (3.,2.) circle (1.5pt);
\draw [fill=black] (4.,1.) circle (1.5pt);
\draw [fill=black] (7.,2.) circle (1.5pt);
\draw [fill=black] (6.,1.) circle (1.5pt);
\end{tikzpicture}
\caption{The $3$-uniform clutter $\mathcal{C}$}
\label{umbrella}
\end{figure}
Note that $\mathcal{C}$ is a decomposable clutter. To see this, we observe that the clutters
\begin{equation*}
\mathcal{C}_1 = \left\{ \{1,2,3\}, \{1,3,4\}, \{2,3,4\}, \{1,2,4\} \right\} \quad \text{and} \quad \mathcal{C}_2 = \left\{ \{1,2,4\}, \{1,2,5\}, \{1,4,5\}, \{2,4,5\} \right\}
\end{equation*}
are decomposable, for they are complete clutters on $4$ vertices. Since $V \left( \mathcal{C}_1 \right) \cap V \left( \mathcal{C}_2 \right) =\{1,2,4\}$ is a clique in both $\mathcal{C}_1$ and $\mathcal{C}_2$, the clutter $\mathcal{C}_3 = \mathcal{C}_1 \cup \mathcal{C}_2$ is a decomposable clutter. Now, let $e_1 =\{3,4\}, e_2=\{2,4\}, A_1 = \{\{2,3,4\}\}$ and $A_2= \{\{1,2,4\}, \{2,4,5\}\}$. Then $\mathcal{C} = \mathcal{C}_3 \setminus A_1 \setminus A_2$ is a decomposable clutter, by definition.

Now let $\Delta$ be a simplicial complex with $\mathcal{F}(\Delta^{[2]}) = \mathcal{C}$. Then $\dim (\Delta) = 2$ and $\mathcal{C} \subseteq \mathcal{F} \left( \Delta \right)$. 
But $\Delta$ can not be a quasi-forest, because for any order  $F_1, F_2, F_3, F_4$ of the elements in $\mathcal{F}(\Delta)\sect \mathcal{C}$,  the facet $F_4$ does not have a free vertex.
\end{ex}

\begin{ex} \label{ex 2}
Let
\begin{align*}
\mathcal{C} = \{ &\{1,2,3\}, \{1,2,4\}, \{1,2,5\}, \{1,2,6\}, \{1,3,4\}, \{1,3,5\}, \{1,4,6\}, \{1,5,6\},\\
 &\{2,3,5\}, \{2,3,6\}, \{2,4,5\}, \{2,4,6\} \}.
 \end{align*}
Then $\mathcal{C}$ is a decomposable clutter. To see this, we observe that by letting
\[
\begin{array}{llll}
&e_1=\{5,6\}, &\quad & A_1= \{\{2,5,6\}, \{3,5,6\}, \{4,5,6\} \} \\
&e_2=\{3,6\}, &\quad & A_2= \{\{1,3,6\}, \{3,4,6\} \} \\
&e_3=\{3,4\}, &\quad & A_3= \{ \{2,3,4\}, \{3,4,5\} \} \\
&e_4=\{4,5\}, &\quad & A_4= \{\{1,4,5\} \}, \\
\end{array}
\]
we have $e_1\in \mathrm{Simp}( \mathcal{C}_{6,3})$ and for $2\leq i\leq 4$, $e_i\in \mathrm{Simp}( \mathcal{C}_{6,3}\setminus A_1\setminus \ldots\setminus A_{i-1})$. Moreover, $\mathcal{C} = \mathcal{C}_{6,3} \setminus A_1 \setminus A_2 \setminus A_3 \setminus A_4$. So, $\mathcal{C}$ is a decomposable clutter. But the ideal $I \left( \bar{\mathcal{C}}\right)^2$ does not have linear resolution (c.f. \cite[p. 284]{Sturmfels}).
\end{ex}


\subsection{Squarefree stable ideals}
For a monomial $u \in S=\mathbb{K}[x_1,\ldots,x_n]$ we set $m(u) = \max\{i: x_i \mbox{ divides } u\}$, and call a (squarefree) monomial ideal {\em $($squarefree$)$  stable}, if for all (squarefree) monomials $u \in I$, and all $i < m(u)$ (such that $x_i$ does not divide $u$) one has $x_i(u/x_{m(u)})\in I$. It is easy to see that the defining property of a (squarefree)  stable ideal needs to be checked only for the set of monomial generators of the  ideal, \cite[Problem~4.1]{HHBook}. The class of (squarefree) stable ideals have linear quotients \cite[Problem~8.8(b)]{HHBook} and hence linear resolution over all fields. In \cite[Theorem~2.5]{NZ} it is proved that the uniform clutters associated to squarefree stable ideals are chordal. The question comes whether these clutters  are decomposable too. 

A (squarefree)  monomial ideal $I$ is called a {\em $($squarefree$)$ lexsegment ideal} if for all $($squarefree$)$ monomials $u\in I$ and all $($squarefree$)$ monomials $v\in \mathbb{K}[x_1,\ldots,x_n]$ with $\deg v = \deg u $ and $v \geq_{lex} u$ one has $v\in  I$.

Let $I$ be a $($squarefree$)$ monomial ideal in $S$. Then $I$ is called {\em $($squarefree$)$ strongly stable} if one has $x_i(u/x_j)\in I$ for all (squarefree) monomials $u \in I$ and all $i < j$ such that $x_j$ divides $u$ (and $x_i$ does not divide $u$).

It is known that (squarefree) lexsegment ideals and  (squarefree) strongly stable ideals are (squarefree) stable (see \cite[Page~103]{HHBook}). Following this fact,   \cite[Theorem~2.5]{NZ} and \cite[Theorem 3.2]{MA} we get:

\begin{cor}\label{stable,lex}
Let $I$ be an equigenerated squarefree monomial ideal which is either  lexsegmet, or  strongly stable or  stable in $\mathbb{K}[x_1,\ldots,x_n]$ and let   $\mathcal{C}$ be the uniform clutter on the vertex set $[n]$ with $I=I(\bar{\mathcal{C}})$. Then  
\begin{itemize}
\item[\rm (i)] \cite[Theorem~2.5]{NZ} $\mathcal{C}$  is a chordal clutter.
\item[\rm (ii)] \cite[Theorem 3.2]{MA} $\mathcal{C}$  is  a simplicial subclutter of $\mathcal{C}_{n,d}$. In particular, $\mathcal{C}$ is a decomposable clutter. 
\end{itemize}
 \end{cor}

The following example shows that not all simplicial subclutters of a complete clutter end in squarefree lexsegmet,  squarefree strongly stable or squarefree stable ideals. They do not always lead to matroidal ideals too. Recall that an equigenerated squarefree monomial ideal $I$ is called matroidal if for each pair $u,v$ in the minimal generating set of $I$, whenever $x_i|u$ and $x_i\not| v$, then there exists $j$ with $x_j|v$ and $x_j\not| u$ such that $x_j(u/x_i)\in I$.  
\begin{ex}
Let $I=I(\bar{\mathcal{C}})$, where $\mathcal{C}$ is the clutter in Example~\ref{ex 2}. Since $I^2$ does not have linear resolution, it follows that $I$ is not matroidal because all powers of a matroidal ideal have linear resolution over all fields \cite[Corollary~12.6.4]{HHBook}. Moreover, $I$ is not stable because for $u=x_3x_5x_6\in I$, $m(u)=6$, and we have $x_1(u/x_6)=x_1x_3x_5\notin I$. Since all lexsegment ideals and all strongly stable ideals are stable, we conclude that $I$ is neither lexsegment nor strongly stable.
\end{ex}



\begin{thebibliography}{99}
\bibitem{MA}
M.~Bigdeli, and A.~A.~Yazdan Pour, \textit{Multigraded minimal free resolution of simplicial subclutters}, preprint.

\bibitem{BHYZ}
M.~Bigdeli, J.~Herzog, A.~A.~Yazdan Pour, and R.~Zaare-Nahandi, \textit{Simplicial orders and chordality}, Journal of Algebraic Combinatorics, Volume \textbf{45}, Issue \textbf{4}, pp.~1021--1039 (2017).

\bibitem{BYZ}
M.~Bigdeli, A.~A.~Yazdan Pour, and R.~Zaare-Nahandi, \textit{Stability of Betti numbers under reduction processes: Towards chordality of clutters}, Journal of Combinatorial Theory Series A, \textbf{145}, pp.~129--149 (2017).

\bibitem{Bjorner}
A.~Bj\"orner, and K.~Eriksson, \textit{Extendable shellability for rank $3$ matroid complexes}, Discrete Mathematics, \textbf{132}, pp. 373--376 (1994).

\bibitem{Bjorner-Wachs}
A.~Bj\"orner, and M.~L.~Wachs, \textit{Shellable nonpure complexes and posets I}, Transactions of the American Mathematical Society,  \textbf{348}, Number \textbf{4} (1996).


\bibitem{Brugesser}
H.~Brugesser, and P.~Mani, \textit{Shellable decompositions of cells and spheres}, Mathematica Scandinavica, \textbf{29.2}, pp.~197--205 (1972).


\bibitem{Klee2}
G.~Danaraj, and V.~Klee, \textit{Shellings of spheres and polytopes}, Duke Mathematical Journal, \textbf{41}, pp.~443--451 (1974).

\bibitem{Klee}
G.~Danaraj, and V.~Klee, \textit{Which spheres are shellable?}, Annals of Discrete Mathematics, \textbf{2}, pp.~33--52 (1978).

\bibitem{Dirac}
G.~A.~Dirac, \textit{On rigid circuit graphs}, Abhandlungen aus dem Mathematischen Seminar der Universit\"at Hamburg, \textbf{38}, pp. 71--76 (1961).


\bibitem{Fr}
R.~Fr\"{o}berg, \textit{On Stanley--Reisner rings}, in: Topics in algebra, Banach Center Publications, \textbf{26} Part 2, pp.~57--70 (1990).

\bibitem{Gallier}
J.~Gallier, \textit{Notes on Convex Sets, Polytopes, Polyhedra, Combinatorial Topology}, available at:  \href{https://arxiv.org/abs/0805.0292}{\texttt{arxiv.org/abs/0805.0292}} (2008).

\bibitem{Xavier}
X.~ Goaoc, P.~Pat\'ak,  Z.~Pat\'akov\'a,  M.~Tancer,  and U.~Wagner,  \textit{Shellability is NP-complete},  Symposium on Computational Geometry 2018: 41:1-41:15.

\bibitem{Hall}
H.~T.~Hall, \textit{Counterexamples in Discrete Geometry}, Ph.D. Thesis, University of California, Berkeley, (2004).

\bibitem{HHBook}
J.~Herzog, and T.~Hibi, \textit{Monomial Ideals}, in: GTM  \textbf{260}, Springer, London, (2011).

\bibitem{HHZ}
J.~Herzog, T.~Hibi, and X.~Zheng, \textit{Dirac's theorem on chordal graphs and Alexander duality}, European Journal of Combinatorics, \textbf{25}, pp. 949--960 (2004).

\bibitem{HHZ2}
J.~Herzog, T.~Hibi, and X.~Zheng, \textit{Monomial ideals whose powers have a linear resolution}, Mathematica Scandinavica, \textbf{95}, pp.~23--32  (2004).

\bibitem{Kleinschmidt}
P.~Kleinschmidt, \textit{Untersuchungen zur Struktur geometrischer Zellkomplexe insbesondere zur Schalbarkeit von p.1.-Sph\"aren und p.l.-Kugeln}, habilitationsschrift, Ruhr-Universit\"at-Bohum (1977).


\bibitem{MNYZ}
M.~Morales, A.~Nasrollah Nejad, A.~A.~Yazdan Pour, and R.~Zaare-Nahandi, \textit{Monomial ideals with $3$-linear resolutions}, Annales de la Facult\'e des Sciences de Toulouse, S\'er. \textbf{6}, 23: (4), pp.~877--891 (2014).

\bibitem{MYZ}
M.~Morales, A.~A.~Yazdan Pour, and R.~Zaare-Nahandi, \textit{Regularity and Free Resolution of Ideals which are Minimal to $d$-linearity}, Mathematica Scandinavica, \textbf{118} (2), pp.~161--182 (2016).


\bibitem{NZ} 
A.~Nikseresht, and R.~Zaare-Nahandi, \textit{On generalization of cycles and chordality to clutters from an algebraic viewpoint}, Algebra Colloquium, \textbf{24}, 611 (2017).

\bibitem{Provan}
J.S. Provan and L.J. Billera,\textit{ Decompositions of simplicial complexes related to diameters of convex polyhedra}, Math. Oper. Res. \textbf{5}, pp. 576--594 (1980).

\bibitem{Simon}
R.~S.~Simon, \textit{Combinatorial properties of cleanness}, Journal of Algebra \textbf{167}, pp.~361--388 (1994).


\bibitem{Stanley}
R. P. Stanley, \textit{Combinatorics and Commutative Algebra}, second ed., Progress in Mathematics, vol. 41, Birkh\"auser Boston Inc., Boston, MA, (1996).

\bibitem{Sturmfels}
B.~Sturmfels, \textit{Four counterexamples in combinatorial algebraic geometry}, Journal of Algebra \textbf{230}, pp.~282--294 (2000).


\bibitem{Vantuyl}
A.~Van Tuyl, and R.~H.~Vilarreal, \textit{Shellable graphs and sequentially Cohen--Macaulay bipartite graphs}, Journal of Combinatorial Theory Series A,  \textbf{115.5}, pp.~799--814 (2008).


\bibitem{Wachs2}
M. L. Wachs, \textit{Obstructions to shellability}, Discrete Comput. Geom. \textbf{22}, no. \textbf{1}, pp. 95--103 (1999), \href{https://arxiv.org/abs/math/9707216}{\texttt{arXiv:math/9707216}}.

\bibitem{Ziegler}
G.~M.~Ziegler, \textit{Shelling Polyhedral $3$-Balls and $4$-Polytopes}, Discrete \& Computational Geometry \textbf{19}, Issue 2, pp.~159--174  (1998).
\end{thebibliography}
\end{document}